\documentclass[12pt]{amsart} 
\usepackage{amssymb}
\usepackage{latexsym}
\usepackage{amsmath}
\usepackage{amsthm}
\usepackage{amsfonts}

\newtheorem{theorem}{Theorem}%[section]

\newtheorem{lemma}{Lemma}
\newtheorem{corollary}{Corollary}
\newtheorem{proposition}{Proposition}

\theoremstyle{definition}

\newtheorem*{acknowledgements*}{Acknowledgements}
\theoremstyle{remark}
\newtheorem{remark}{Remark}
\newtheorem{note}[theorem]{Note}

\usepackage{color}

\usepackage{ulem} %For crossing out text, use \textcolor{red}{\sout{TEXT}}

\newcommand{\R}{\mathbb R}
\newcommand{\N}{\mathbb N}

\newcommand\subsetsim{\mathrel{\substack{
  \textstyle\subset\\[-0.ex]\textstyle\sim}}}

\newcounter{obs}

\title
[Rosenthal's space revisited]
{Rosenthal's space revisited}

\author[S.V. Astashkin]{Sergey V. Astashkin}

\address{Department of Mathematics, 
Samara National Research University, Moskovskoye shosse 34, 443086,
Samara, Russia}
\email{astash56@mail.ru}

\thanks{The work of the first author was supported by the Ministry of Science and Higher Education of the Russian Federation (project 1.470.2016/1.4).}
\thanks{The second author acknowledges the support  of PGC2018-096504-B-C31, FQM-262 and Feder-US-1254600.}

\author[G.P. Curbera]{Guillermo P. Curbera}

\address{Facultad de Matem\'aticas \& Instituto de Matem\'aticas (IMUS),
Universidad de Sevilla, 
Calle Tarfia s/n,  Sevilla 41012, Spain}
\email{curbera@us.es}

\date{\today}


\begin{document}

\begin{abstract}
Let $E$ be a rearrangement invariant (r.i.) function space on $[0,1]$, and let $Z_E$ consist of all measurable functions $f$ on $(0,\infty)$ such that $f^*\chi_{[0,1]}\in E$ and $f^*\chi_{[1,\infty)}\in L^2$. We  reveal close connections between properties of the generalized Rosenthal's space, corresponding to the space $Z_E$, and the behaviour of independent symmetrically distributed random variables  in $E$. The results obtained are applied to consider the problem of the existence of isomorphisms between r.i.\ spaces on $[0,1]$ and $(0,\infty)$. Exploiting  particular properties of disjoint sequences, we identify a rather wide new class of r.i.\ spaces on $[0,1]$ ``close'' to $L^\infty$, which fail to be isomorphic to r.i.\ spaces on $(0,\infty)$. In particular, this property is shared by the Lorentz spaces $\Lambda_2(\log^{-\alpha}(e/u))$, with $0<\alpha\le 1$. 
\end{abstract}

\maketitle

%%%%%%%%%%%%%%%%%%%%%%%%%%%%%%%%%%

\section{Introduction}
\label{Intro}

%%%%%%%%%%%%%%%%%%%%%%%%%%%%%%%%%%%

Let $p>2$. Given any sequence $w=(w_n)_{n=1}^\infty$ of positive scalars such that
\begin{equation}\label{1}
\sum_{n=1}^\infty w_n^{2p/(p-2)}=\infty \quad\text{and}\quad \lim_{n\to\infty}w_n=0,
\end{equation}
we define $X_{p,w}$ to be the space of  all sequences $(a_n)_{n=1}^\infty$ of scalars
satisfying
$$
\sum_{n=1}^\infty |a_n|^p<\infty \quad\text{and}\quad \sum_{n=1}^\infty |a_n|^2w_n^2<\infty,
$$
under the norm
$$
\|(a_n)_{n=1}^\infty\|:=\max\big\{\|(a_n)_{n=1}^\infty\|_p,\|(a_nw_n)_{n=1}^\infty\|_2\big\},
$$
where $\|(a_n)_{n=1}^\infty|\|_r=\big(\sum_{n=1}^\infty|a_n|^r\big)^{1/r}$, $1\le r<\infty$.
Note that, up to isomorphism, the definition of the space $X_{p,w}$ does not depend
on the sequence $w$, i.e., $X_{p,w}\approx  X_{p,w'}$, as long as both $w$ and $w'$ satisfy 
 \eqref{1}; \cite[Theorem 13]{rosenthal}. Hence, we can denote the space $X_{p,w}$ simply by $X_p$.

The space $X_p$, introduced by Rosenthal in 1970 
(see \cite{rosenthal}), turned out 
to be very useful when studying the geometric structure of $L^p$-spaces. Specifically,
$X_p$ is isomorphic to the complemented subspace of $L^p$ 
spanned by a 
certain sequence of independent 3-valued symmetrically 
distributed random variables (r.v.'s)  \cite[p. 282--283]{rosenthal}. Moreover, for each $p>2$ 
and an arbitrary sequence $\{f_n\}_{n=1}^\infty\subseteq L^p[0,1]$ 
of mean zero independent r.v.'s,  the mapping $T\colon X_p\to L^p$, defined by
$$
T(a_n):=\sum_{n=1}^\infty a_nf_n,
$$
is an isomorphic embedding; \cite[Theorem~3 and p.~280]{rosenthal}.

Later on, Johnson, Maurey, Schechtman, and Tzafriri introduced, in the memoir
\cite{JMST} (see p.~218), the following generalized space of 
Rosenthal type. Let $Y$ be an arbitrary rearrangement invariant (r.i.) space on $(0,\infty)$. Suppose that $\{A_n\}_{n=1}^\infty$ is a sequence of 
disjoint measurable subsets of $(0,\infty)$ of positive measure such that 
\begin{equation}\label{2}
m(A_n)\le1,\quad m(A_n)\to0\; (n\to\infty),\quad 
\sum_{n=1}^\infty m(A_n)=\infty
\end{equation}
($m$ is the Lebesgue measure). Then, the space $\widetilde{\mathcal{U}}_Y$ is defined as a Banach space which
is isomorphic to the closed linear span of the sequence 
$\{\chi_{A_n}\}_{n=1}^\infty$ in $Y$. It is worth to note that, up to isomorphism,
the latter span does not depend on the particular choice of sequence $\{A_n\}_{n=1}^\infty$
satisfying conditions \eqref{2} \cite[Lemma 8.7]{JMST}.
The sequence $\{\|\chi_{A_n}\|_Y^{-1}\chi_{A_n}\}_{n=1}^\infty$,
clearly is equivalent to an unconditional basis in $\widetilde{{\mathcal{U}}}_Y$.
Moreover, if the space $Y(0,1)$ is not equal to $L^\infty(0,1)$  up to an equivalent renorming, $\widetilde{{\mathcal{U}}}_Y$ is isomorphic to a complemented subspace of $Y$.

To establish a link between the concepts so far introduced, recall a further  important definition from \cite{JMST} (see also \cite[\S2f]{lindenstrauss-tzafriri}). Given a r.i.\ space $E$ on $[0,1]$,
we define the r.i.\ space $Z_E$ on $(0,\infty)$ consisting of all measurable functions $f$ on $(0,\infty)$ such that  
\begin{equation}\label{3}
\|f\|_{Z_E}:=\|f^*\chi_{[0,1]}\|_E+\|f^*\chi_{[1,\infty)}\|_{L^2}<\infty,
\end{equation}
where $f^*$ is the non-increasing left-continuous rearrangement of $|f|$ (observe that
$\|\cdot\|_{Z_E}$ is a quasinorm, which is equivalent to a norm;
\cite[Theorem 2.f.1]{lindenstrauss-tzafriri}).

Then, denoting ${\mathcal{U}}_E:=\widetilde{{\mathcal{U}}}_{Z_E}$, it can be checked that 
Rosenthal's space $X_p$ 
coincides, up to equivalence of norms, with the space 
${\mathcal{U}}_{L^p[0,1]}$ (in particular, we choose $w_n=m(A_n)^{1/2-1/p}$, see details in
\cite[p. 221]{JMST}).

The main aim of this paper is to reveal close connections between properties of the space ${\mathcal{U}}_E$ and the behaviour of independent r.v.'s  in the 
corresponding r.i.\ space $E$.

Let $E$ be a r.i.\ space on $[0,1]$. 
%Given a sequence 
%of independent symmetrically distributed r.v.'s  $\{x_n\}_{n=1}^\infty\subseteq E$, we
%associate a sequence $\{\overline{x}_n\}_{n=1}^\infty$ of disjoint copies of them, supported on $(0,\infty)$. Thus, 
According to \cite[Theorem 1]{JS}, if $L^q[0,1]\subseteq E$ for some $q<\infty$, then there is a constant
$C=C(q)>0$ such that for every sequence $\{x_n\}_{n=1}^\infty$ 
of independent symmetrically distributed r.v.'s from $E$ we have
\begin{equation}\label{4}
\bigg\|\sum_{n=1}^\infty x_n\bigg\|_E
\le C\bigg\|\sum_{n=1}^\infty \overline{x}_n\bigg\|_{Z_E},
\end{equation}
where the sequence $\{\overline{x}_n\}_{n=1}^\infty$ consists of pairwise disjoint functions defined on $(0,\infty)$ such that $\overline{x}_n$ and $x_n$ are equimeasurable for each $n=1,2,\dots$ (it is worth to mention that the opposite inequality holds in every r.i.\ space $E$). More recently, in \cite{AS-08} (for a simpler proof see \cite[Theorem 25]{AS-10}),
the latter result was sharpened; it was proved that inequality \eqref{4} holds in every r.i.\ space $E$ that has the so-called Kruglov property (for definitions see the next section). Observe, for instance, that  the exponential Orlicz space $\text{Exp}L^p$,
generated by an Orlicz function equivalent to the function  $e^{u^p}$ for large $u>0$, has the  Kruglov property if and only if $0<p\le1$ (clearly, $\text{Exp}L^p$ does not contain $L^q$ for any $q<\infty$).

In the first part of the paper we show that inequality \eqref{4} is fulfilled for the class of independent symmetrically distributed r.v.'s in a r.i.\ space $E$ with the Fatou property  whenever a similar estimate holds for the subspace ${\mathcal{U}}_E$ of $Z_E$.
More precisely, if  $\{A_n\}_{n=1}^\infty$ is a sequence of 
disjoint measurable subsets of $(0,\infty)$ satisfying \eqref{2}, then inequality \eqref{4} is a consequence of the  following much weaker condition: there is a constant $C>0$ such that for every set $S\subseteq\N$, with $\sum_{n\in S}m(A_n)\le1$, and all $a_n\in\R$
\begin{equation}\label{4weak}
\bigg\|\sum_{n\in S} a_nu_n\bigg\|_E\le C \bigg\|\sum_{n\in S} a_n \chi_{A_n}\bigg\|_{Z_E},
\end{equation}
where $u_n$ are independent symmetrically distributed functions, equimeasurable with the characteristic functions $\chi_{A_n}$ (see Theorem~\ref{t1}). Moreover, we prove in Theorem \ref{t2} that estimate \eqref{4weak} combined with a certain geometrical property of the subspace $[u_n]$ of a r.i.\ space $E$ ensures that $E\approx Z_E$.

Next, we apply the results obtained to consider the problem of the 
existence of isomorphisms between r.i.\ spaces on $[0,1]$ and $(0,\infty)$, which was first posed by Mityagin in \cite{mitiagin}. This and other closely related questions were intensively studied in the memoir \cite{JMST} (see also \cite{lindenstrauss-tzafriri}), by using the approach based on a construction of 
the stochastic integral with respect to a symmetrized Poisson process. In particular, it was shown that a r.i.\ space $E$ is isomorphic to the space $Z_E$ whenever $0<\alpha_E\le \beta_E<1$, where $\alpha_E$ and 
$\beta_E$ are the Boyd indices of $E$ (see \cite[Theorem~8.6]{JMST} or
\cite[Theorem~2.f.1]{lindenstrauss-tzafriri}). Later on, in \cite{A-10}, this result was improved: it turned out that non-triviality of the Boyd indices of $E$ can be replaced with a weaker condition that both spaces $E$ and its K\"othe dual $E'$ have the Kruglov property.

However, there exist r.i.\ spaces on $[0,1]$ which are not isomorphic to r.i.\ spaces on $(0,\infty)$.
Roughly speaking, this property is shared by some r.i.\ spaces, which are located ``very close'' to the extreme r.i.\ spaces on $[0,1]$, $L^1$ and $L^\infty$. For instance, this holds for the Orlicz space $L_{F_\alpha}$, $0<\alpha<1/2$, where $F_\alpha(u)$ is an Orlicz function equivalent to the function $u \log ^\alpha u$ for large $u>0$ \cite[p.~235]{JMST}. Observe that the only r.i.\ space on $(0,\infty)$, which can be isomorphic to $L_{F_\alpha}$ is the space $Z_{L_{F_\alpha}}$ (see \cite[Corollary~8.15 and subsequent remarks]{JMST}). In such a case the result follows easily from the fact that either the space 
$E$ itself or its dual $E^*$ does not contain sequences equivalent to the unit vector $\ell^2$-basis, because both spaces $\mathcal{U}_E$ and $Z_E$, clearly, contain such sequences. Indeed, if we assume that $L_{F_\alpha}\approx  Z_{L_{F_\alpha}}$, with $0<\alpha<1/2$, then it would  imply by duality that  $\text{Exp}L^{1/\alpha}\approx  Z_{\text{Exp}L^{1/\alpha}}$ (see Lemma \ref{l1}). But this is a contradiction because the exponential Orlicz space $\text{Exp}L^{r}$, for $r>2$, contains no sequences equivalent to the unit vector $\ell^2$-basis (for instance, this is a consequence of Proposition~\ref{l2} with its proof combined with the well-known fact that any disjoint sequence in $\text{Exp}L^{r}$ contains a subsequence equivalent to the unit vector $c_0$-basis; see e.g. \cite{T}).
%The latter is a consequence of the well-known 
%Kadec-Pe\l czy\'nski alternative \cite{KP}, the fact that every sequence of pairwise disjoint functions
%$\{x_n\}_{n=1}^\infty\subseteq \text{Exp}L^{r}$, with $\|x_n\|=1$, contains a subsequence
%equivalent to the unit vector $c_0$-basis, and a result of Raynaud \cite{raynaud} (see also Lemma \ref{l2}).

Here, we present more non-trivial examples of r.i.\ spaces $E$ of such a sort,  showing that even the existence of complemented subspaces isomorphic to $\ell^2$ does not guarantee that ${\mathcal{U}}_E$ is isomorphically embedded into $E$. Specifically, exploiting particular properties of disjoint sequences, we identify a rather wide new class of r.i.\ spaces on $[0,1]$ ``close'' to $L^\infty$, which fail to be isomorphic to r.i.\ spaces on $(0,\infty)$ (see Theorems \ref{3}, \ref{t4} and \ref{t4a}). Furthermore, in Corollary~\ref{c2}, we provide examples of Lorentz spaces $\Lambda_2(\varphi)$ containing plenty of complemented subspaces isomorphic to $\ell^2$, but without subspaces isomorphic to the corresponding Rosenthal's spaces and not isomorphic to r.i.\ spaces on $(0,\infty)$. In particular, these properties are shared by the Lorentz spaces $\Lambda_2(\log^{-\alpha}(e/u))$, with $0<\alpha\le 1$ (see Corollary~\ref{c4}).

In the concluding part of the paper, in Theorem~\ref{t5a}, we prove a partial result related to the problem if the Kruglov property of a r.i.\ space $E$  is a necessary condition for the existence of an isomorphic embedding $T\colon {\mathcal{U}}_E\to E$. We consider the case when $T$ sends the basis functions $\chi_{A_n}$, $n=1,2,\dots$, of $Z_E$ to some independent symmetrically distributed r.v.'s  in $E$.

%%%%%%%%%%%%%%%%%%%%%%%%%%%%%%%

\section{Preliminaries}

%%%%%%%%%%%%%%%%%%%%%%%%%%%%%%%%

\subsection{Rearrangement invariant spaces}
\label{prel1}
For a detailed account of basic properties of rearrangement invariant spaces, we refer to the monographs \cite{BS,krein-petunin-semenov,lindenstrauss-tzafriri}.

Let $I=[0,1]$ or $(0,\infty)$. A Banach lattice $E$ on $I$ is said to be a {\it rearrangement invariant} (in brief, r.i.) (or {\it symmetric}) space  
if from the conditions: functions $x(t)$ and $y(t)$ are {\it equimeasurable}, i.e.,
$$
m\{t\in I:\,|x(t)|>\tau\}=m\{t\in I:\,|y(t)|>\tau\}\;\;\mbox{for all}\;\;\tau>0$$
and $y\in E$ it follows $x\in E$ and $\|x\|_E=\|y\|_E$ (throughout, $m$ denotes the Lebesgue measure). 

In particular, every measurable on $I$ function $x(t)$ is equimeasurable with the non-increasing, right-continuous rearrangement of $|x(t)|$ given by
$$
x^*(t):=\inf\{~\tau>0:\,m\{s\in I:\,|x(s)|>\tau\}\le t~\},\quad t>0.$$

We note that for any r.i.\ space $E$ on $[0,1]$ we have:
$L^\infty [0,1]\subseteq E\subseteq L^1[0,1].$ Denote by $E_0$ the closure of $L^\infty[0,1]$ in the r.i.\ space $E$ on $[0,1]$ (the {\it separable part} of $E$). The space $E_0$ is r.i., and it is separable if $E\ne L^\infty$.
The fundamental function $\phi_E$ of a symmetric space $E$ is defined by  $\phi_E(t):=\|\chi_{[0,t]}\|_E$, $t>0$. In what follows, $\chi_A$ is the characteristic function of a set $A$. The function $\phi_E$ is quasi-concave, that is, it is nonnegative and increases, $\phi_X(0)=0$, and the function $\phi_E(t)/t$ decreases. Without loss of generality, we will assume that $\|\chi_{[0,1]}\|_E=1$ for every r.i.\ space $E$.

It is well known that the dilation operator $\sigma_\tau x(t):=x(t/\tau)\chi_{[0,\min(1,\tau)]}(t)$, $0\le t\le1$, is bounded on 
every r.i.\ space $E$ on $[0,1]$ and $\|\sigma_\tau\|_{E\to E}\le\max(1,\tau)$
(see e.g.\ \cite[Ch.II, \S4.3]{krein-petunin-semenov}). The numbers $\alpha_E$ and $\beta_E$ given by
\[
 \alpha_E:=\lim\limits_{\tau\to
 0}\frac{\ln\|\sigma_\tau\|_E}{\ln\tau},\quad
 \beta_E:=\lim\limits_{\tau\to
 \infty}\frac{\ln\|\sigma_\tau\|_E}{\ln\tau}
\]
satisfy the inequalities $0\le  \alpha_E\le\beta_E\le 1$ and are called the {\it Boyd indices} of $E$.

The {\it K\"othe dual} $E'$ of a r.i.\ space $E$ on $I$ consists of all measurable functions $y$ such that
$$
\Vert y\Vert _{E'}:= \sup \Big\{\int _I\vert
x(t)y(t)\vert\,dt:\ x\in E,\ \Vert x \Vert _{{E}}\leq 1\,\Big\}
<\infty.
$$
If $E^*$ denotes the Banach dual of $E$, then $E'\subset
E^{*}$ and $E'=E^{*}$ if and only if $E$ is separable. A r.i.\ space $E$ on $I$ is said to have the {\it Fatou property} if whenever $\{x_n\}_{n=1}^\infty\subseteq E$ and $x$ measurable on $[0,1]$
satisfy $x_n\to x$ a.e. on $I$ and $\sup _{n=1,2,\dots}\Vert x_n\Vert _{{E
}} <\infty $, it follows  that $x\in E$ and $\Vert x\Vert
_{{E}}\leq \liminf _{n\to \infty }\Vert x_n\Vert _{{E}}$. It is
well-known that a r.i.\ space $E$ has the Fatou property if and only
if the natural embedding of $E$ into its K\"othe bidual
$E''$ is a surjective isometry.

\bigskip
%Let us recall some classical examples of r.i.\ spaces on $[0,1]$.
%Denote by $\Omega$ the set of all increasing concave functions on
%$[0,1]$ with $\varphi(0)=\varphi(+0)=\lim\limits_{t\to
%0}\frac{t}{\varphi(t)}=0$. 

An important example of r.i.\ spaces are the Orlicz spaces. Let $\Phi$ be an Orlicz function, i.e., increasing convex function on $[0, \infty)$ such that $\Phi (0) = 0$. Then, the {\it Orlicz space} $L_{\Phi}:=L_{\Phi}(I)$ consists of all measurable on $I$ functions $x$ such that the Luxemburg--Nakano norm
$$
\| x \|_{L_{\Phi}}: = \inf \{\lambda > 0 \colon \int_I \Phi(|x(t)|/\lambda) \, dt \leq 1 \}
$$
is finite (see e.g. \cite{KR61}). In particular, if $\Phi(u)=u^p$, $1\le p<\infty$, then $L_\Phi=L^p$. If $\Phi(u)$ is equivalent for large $u>0$ to the function $e^{u^p}$, $p>0$, we obtain the exponential Orlicz space $\text{Exp}\,L^{p}[0,1]$.

Every increasing concave function on $[0,1]$, with $\varphi(0)=0$, and $1\le q<\infty$ generate the {\it Lorentz} space $\Lambda_q(\varphi)$ endowed with the norm
\[
\|x\|_{\Lambda_q(\varphi)}:=\Big(\int\limits_0^1 x^*(t)^q\,d\varphi(t)\Big)^{1/q}.\]
%We set $\Lambda(\varphi):=\Lambda_1(\varphi)$.

\vskip 0.2cm

\subsection{The Kruglov property and comparison of sums of independent finctions and their disjoint copies in r.i.\ spaces.}
\label{prel2}
Let $f$ be a measurable function on $[0,1]$. Denote by $\pi(f)$ the random variable (in brief, r.v.) $\sum_{i=1}^Nf_i,$
where $f_i$ are independent copies of $f$ (that is, independent
r.v.'s equidistributed with $f$) and $N$ is a r.v. independent of the sequence $\{f_i\}$ and having the Poisson distribution with parameter 1. 
%It is not hard to check that the characteristic function $\theta_{\pi(f)}(t)$ of
 %$\pi(f)$ is equal to the function $\exp\left(\theta_{f}(t)-1\right)$ for all $t\in\mathbb{R},$
%where $\theta_{f}$ is the characteristic function of the random variable $f$.
The following property has its origin in Kruglov's paper \cite{K} and was actively studied and used by Braverman \cite{B}. We say that a r.i.\ space $E$ on $[0,1]$ has the {\it Kruglov property} if the relation $f\in E$ implies that $\pi(f)\in E.$

Roughly speaking, a r.i.\  space $E$ has the Kruglov property if it
is located sufficiently ``far away''  from the space $L^\infty.$  In
particular, if $E$ contains $L^p$ with some $p<\infty$, then $E$ has the Kruglov property.
However, the latter condition is not necessary; for instance, the exponential Orlicz space $\text{Exp}\,L^{p}$ has the Kruglov property if and only if $0<p\le 1$ (see \cite[\S\,2.4]{B}, \cite{AS1}), but clearly $\text{Exp}\,L^{p}$ does not contain $L^q$ with any $p>0$ and $1\le q<\infty$.

The Kruglov property is closely related to the famous Rosenthal
inequality \cite{rosenthal} and more generally to the problem of the comparison of sums of independent functions and their disjoint copies in r.i.\ spaces. 

Let $E$ be a r.i.\ space on $[0,1]$. As was already mentioned in Section~\ref{Intro}, by \cite[Theorem 1]{JS}, if $L^q[0,1]\subseteq E$ for some $q<\infty$, then the inequality \eqref{4} holds for some constant $C=C(q)>0$ and for each sequence of independent symmetrically distributed functions $\{x_n\}_{n=1}^\infty\subset E$. Here, $\bar{x}_n$ are
disjoint copies of $x_n$ defined on the semi-axis $[0,\infty)$ (for
instance, we may take $\bar{x}_n(t)=x_n(t-n+1)\chi_{[n-1,n)}(t),$
$n=1,2,\dots$). We will refer such a sequence $\{\bar{x}_n\}$ as a {\it disjointification} of the sequence $\{{x}_n\}$. Using an operator approach initiated in \cite{AS1} (see also \cite{AS-10}), Astashkin and Sukochev have showed that inequality \eqref{4} holds for a wider class of r.i.\ spaces with the above-defined Kruglov property. 

It is easy to check that the above r.v. $\pi(f)$ is equidistributed with the sum
$$
Kf(t):=\sum_{n=1}^\infty\sum_{i=1}^n f_{n,i}(t)\chi_{E_n}(t),\;\;0\le t\le 1,$$
where $E_n$ are disjoint subsets of $[0,1],$ $m(E_n)=1/(en!)$, $n=1,2,\dots$, and $f_{n,i}$ are functions identically distributed with $f$, $i=1,\dots,n$, $n=1,2,\dots$ such that $f_{n,1},\dots,f_{n,n},\chi_{E_n}$ are independent for each positive integer $n$. It turns out that the above mapping $K$ can be treated as a linear operator defined on suitable r.i.\ spaces (see \cite[p.~1029]{AS-10}). Moreover, given a r.i.\  space $E$ on $[0,1]$, the space $E$ has the Kruglov property if and only if the operator $K$ is bounded in $E$. For this reason, $K$ is called the {\it Kruglov operator}. 

We will say that subsets $F_n$ of $[0,1]$, $n=1,2,\dots$, are {\it independent} if the characteristic functions $\chi_{F_n}$, $n=1,2,\dots$, are independent on $[0,1]$.

Standard Banach space notation is used throughout. In particular, $X\approx Y$, where $X$ and $Y$ are  Banach spaces, means that $X$ and $Y$ are isomorphic. We will write $Y\subsetsim X$ if there is an isomorphic embedding $T\colon Y\to X$. The notation $f\asymp g$ will mean that there exists a constant $C>0$ not depending on the arguments of the quantities (norms) $f$ and $g$ such that $C^{-1}{\cdot}f\le g\le C{\cdot}f$. Finally, in what follows, $C$, $c$ etc. denote constants whose value may change from line to line.

\vskip 0.2cm

%%%%%%%%%%%%%%%%%%%%%%%%%%%%%%%

\section{Rosenthal's space ${\mathcal{U}}_E$ and comparison of sums of independent finctions and their disjoint copies in r.i.\ spaces.}

%%%%%%%%%%%%%%%%%%%%%%%%%%%%%%%%

Let $\{A_n\}_{n=1}^\infty$ be an arbitrary (fixed) sequence of disjoint measurable 
subsets of $(0,\infty)$ satisfying conditions \eqref{2}. 
Denote by $u_n$ independent symmetrically distributed r.v.'s
supported on $[0,1]$ and equimeasurable with the characteristic functions $\chi_{A_n}$, $n=1,2,\dots$. As it was mentioned in Section~\ref{Intro}, if a r.i.\  space $E$ has the Kruglov property (see Section~\ref{prel2}), then there is a constant $C>0$ such that for any sequence $\{x_n\}_{n=1}^\infty$ of independent symmetrically distributed r.v.'s from $E$ inequality \eqref{4} holds. Clearly, then the above r.v.'s $u_n$, $n=1,2,\dots$, satisfy condition \eqref{4weak}.  
%In particular, clearly we have
%
%\begin{equation}\label{5}
%\bigg\|\sum_{n=1}^\infty a_nu_n\bigg\|_E\le C \bigg\|\sum_{n=1}^\infty a_n %\chi_{A_n}\bigg\|_{Z_E}.
%\end{equation}
%
In this section, assuming that a r.i.\ space $E$ has the Fatou property, we prove the converse non-trivial implication: from \eqref{4weak} it follows \eqref{4}. Moreover, starting with this result we will show that estimate \eqref{4weak} combined with a geometrical property of the closed linear span
$[u_n]$ in $E$ implies that $E\approx Z_E$.

First, we consider independent r.v.'s $v_n$, $n=1,2,\dots$, which are  identically distributed with the characteristic functions $\chi_{A_n}$, $n=1,2,\dots$

%If the Kruglov operator $K$ is bounded in an r.i.s.\ $E$ then \eqref{4} holds for any sequence $\{f_n\}$ if i.s.r.v.s  or nonnegative i.r.v.s. Then, clearly, we have 
%%
%\begin{equation}\label{5}
%\bigg\|\sum_{n=1}^\infty a_nu_n\bigg\|_E\le C \bigg\|\sum_{n=1}^\infty a_n \chi_{A_n}\bigg\|_{Z_E}
%\end{equation}
%%
%and
%%
%\begin{equation}\label{6}
%\bigg\|\sum_{n=1}^\infty a_nv_n\bigg\|_E\le C \bigg\|\sum_{n=1}^\infty a_n \chi_{A_n}\bigg\|_{Z_E}.
%\end{equation}
%%
%Now, we prove a converse statement: each of the inequalities 
%\eqref{5} and{\eqref{6} guarantees that \eqref{4} holds. 
%Moreover, we show that the answer to the question if 
%$E\approx Z_E$ or not depends only on the properties of the Rosenthal's space $U_E$.
%
%
%Obviously, if \eqref{6} holds, so does \eqref{5}. 
%We begin by proving the converse implication. 

%%%%%%%%%%%%%%%%%%%%%%%%%%%%%%

\begin{proposition}\label{p2}
Let  $E$ be a r.i.\ space  on $[0,1]$. Suppose that there exists $C>0$ 
such that for every set $S\subseteq\N$ such that $\sum_{n\in S} m(A_n)\le1$ 
and all $a_n\in\R$, $n\in S$, we have
\begin{equation}\label{13}
\bigg\|\sum_{n\in S} a_nv_n\bigg\|_E
\le  C\, \bigg\|\sum_{n\in S} a_n \chi_{A_n}\bigg\|_{Z_E}.
\end{equation}
Then, the Kruglov operator $K$ is bounded from $E$ into $E''$.
\end{proposition}

%%%%%%%%%%%%%%%%%%%%%%%%%%%%%%

\begin{remark}\label{r1}
Clearly, from the condition $\sum_{n\in S} m(A_n)\le1$ and definition of the norm in $Z_E$ (see \eqref{3}) it follows that \eqref{13} can be equivalently rewritten as
\begin{equation*}\label{13'}
\bigg\|\sum_{n\in S} a_nv_n\bigg\|_E\le  
C\, \bigg\|\sum_{n\in S} a_n \chi_{A_n'}\bigg\|_{E},
\leqno{(6')}
\end{equation*}
where sets $A_n'\subseteq[0,1]$ are pairwise disjoint and $m(A_n')=m(A_n)$, $n=1,2,\dots$
\end{remark}

%%%%%%%%%%%%%%%%%%%%%%%%%%%%%%

\begin{proof}
According to \cite[Theorem 22(i)]{AS-10}, it suffices to prove that there is a constant $C'>0$ such that for every sequence $\{x_n\}_{n=1}^l\subseteq E$ of independent functions, with $\sum_{n=1}^l m(\{t:x_n(t)\not=0\})\le1$, we have
\begin{equation}\label{14}
\bigg\|\sum_{n=1}^l x_n\bigg\|_E 
\le  C'\, \bigg\|\sum_{n=1}^l \overline{x}_n\bigg\|_E,
\end{equation}
where $\{\overline{x}_n\}_{n=1}^l$ is a disjointification of the sequence $\{{x}_n\}_{n=1}^l$ (we may and will assume that all the functions $\overline{x}_n$ are supported on $[0,1]$).
Moreover, without loss of generality, we suppose that $x_n\ge0$, $n=1,\dots,l$. For arbitrary $\varepsilon>0$ and $k\in\N$ we set
$$
G_n^k:=\big\{t: \varepsilon(k-1)<x_n(t)\le \varepsilon k\big\},
\quad 
F_n^k:=\big\{t: \varepsilon(k-1)<\overline{x}_n(t)\le \varepsilon k\big\}.
$$
Observe that,
%$$
%\text{supp } x_n= \bigcup_{k=1}^\infty G_n^k,
%\quad 
%\text{supp }\overline{x}_n= \bigcup_{k=1}^\infty F_n^k,
%$$
for every $n=1,2,\dots,l$, the sets $G_n^k$, $k=1,2,\dots$ 
(resp. $F_n^k$, $k=1,2,\dots$, $n=1,2,\dots,l$) are pairwise disjoint. Due to  properties \eqref{2}, for each $n=1,\dots,l$ and all $k\in\N$, we can find pairwise disjoint sets $S_n^k\subseteq \N$ such that 
\begin{equation}\label{15}
m(G_n^k)=m(F_n^k)=\sum_{i\in S_n^k}m(A_i).
\end{equation}
Define now the step-functions
$$
y_n:=\sum_{k=1}^\infty \varepsilon k\cdot\chi_{G_n^k}
\quad\text{and}\quad
z_n:=\sum_{k=1}^\infty \varepsilon k\cdot\chi_{F_n^k},
\quad n=1,\dots,l.
$$
Clearly, the functions $y_n$, $n=1,\dots,l$, are independent and
\begin{equation}\label{16}
x_n\le y_n,\quad n=1,\dots,l.
\end{equation}
Fix $n=1,2,\dots,l$. Then, the sets $G_n^k$, $k\in\N$, are pairwise disjoint. Therefore, thanks to \eqref{15}, we can represent the set $G_n^k$ in the form
$$
G_n^k=\bigcup_{i\in S_n^k} G_n^{k,i},\quad k\in\N,
$$
where $G_n^{k,i}\subseteq[0,1]$ are pairwise disjoint for all $i\in S_n^k$, $k\in\N$, and $m(G_n^{k,i})=m(A_i)$, $i\in S_n^k$. Furthermore, we see that
$$
y_n=\sum_{k=1}^\infty \varepsilon k\sum_{i\in S_n^k}\chi_{G_n^{k,i}},
\quad n=1,\dots,l.
$$

Next, denote by $v_n^{k,i}$ independent copies of the characteristic functions $\chi_{G_n^{k,i}}$, $i\in S_n^k$, $k\in\N$, $n=1,2,\dots,l$.
Then, for each $n=1,2,\dots,l$, the sequence $\{\varepsilon k\cdot\chi_{G_n^{k,i}}\}_{i\in S_n^k,k\in\N}$ is a disjointification of the sequence $\{\varepsilon k\cdot v_n^{k,i}\}_{i\in S_n^k,k\in\N}$ (see Section~\ref{prel2}). Therefore, if
$$
f_n:=\sum_{k=1}^\infty \varepsilon k\sum_{i\in S_n^k}v_n^{k,i},\;\;n=1,2,\dots,l,
$$
then, by \cite[Proposition 1]{HMs} (see also \cite[Proposition 7]{AS-10}),
we have 
\begin{align}\label{17}
m(\{t:y_n(t)>\tau\})
&\le 2
m(\{t: \sup_{k\in\N,\atop i\in S^k_n} \varepsilon k\cdot v_n^{k,i}(t)>\tau\})
\nonumber
\\ & \le 2
m(\{t:f_n(t)>\tau\}).
\end{align}
Since $y_n$, $n=1,\dots,l$ (respectively, $f_n$, $n=1,\dots,l$)
are nonnegative independent r.v.'s,  the sequence $\{y_n\}_{n=1}^l$ (resp. 
$\{f_n\}_{n=1}^l$) has the same distribution as the sequence
$\{y_n^*(t_n)\}_{n=1}^l$ (resp. $\{f_n^*(t_n)\}_{n=1}^l$), which is defined on the probability space
$([0,1]^l,\prod_{n=1}^lm_n)$ (for each $n=1,\dots,l$, $m_n$ is the Lebesgue
measure on $[0,1]$). Furthermore, from \eqref{17} and definition of the rearrangement of a measurable function it follows that
\begin{equation}\label{new}
\sigma_{1/2}(y_n^*)(t_n)=y_n^*(2t_n)\le f_n^*(t_n),
\quad 0\le t_n\le1/2.
\end{equation}
It can easily be checked that the functions $\sigma_{1/2}y_n$, $n=1,2,\dots,l$,
are independent on the interval $[0,1/2]$. Indeed, for
arbitrary intervals $I_1,\dots,I_l$ of $\R$ we have
\begin{align*}
m(\{t\in[0,1/2]: (\sigma_{1/2}y_j)(t)\in I_j,\, &j=1,\dots,l\})
\\ &=
m(\{t\in[0,1/2]: y_j(2t)\in I_j, j=1,\dots,l\})
\\ & =
\frac12m(\{t\in[0,1]: y_j(t)\in I_j, j=1,\dots,l\})
\\ & =
\frac12\prod_{j=1}^l m(\{t\in[0,1]: y_j(t)\in I_j\})
\\ & =
\frac{1}{2^{l+1}}\prod_{j=1}^l m(\{t\in[0,1/2]: y_j(2t)\in I_j\})
\\ & =
\frac{1}{2^{l+1}}\prod_{j=1}^l m(\{t\in[0,1/2]: (\sigma_{1/2}y_j)(t)\in I_j\})
\end{align*}
Hence, from \eqref{new}, we have
\begin{align*}
\bigg\|\sigma_{1/2}\Big(\sum_{n=1}^l y_n\Big)\bigg\|_E & 
= \bigg\|\sum_{n=1}^l \sigma_{1/2}(y_n)\bigg\|_{E}
\\ &
= \bigg\|\sum_{n=1}^l (\sigma_{1/2}y_n)^*(t_n)\bigg\|_{E([0,1]^l)}
\\ & \le
\bigg\|\sum_{n=1}^l f_n^*(t_n)\bigg\|_{E([0,1]^l)}
\\ & =
\bigg\|\sum_{n=1}^l f_n\bigg\|_{E}.
\end{align*}
Since $\|\sigma_\tau\|_{E\to E}\le\max(1,\tau)$ (see Section~\ref{prel1} or \cite[Ch.II, \S4.3]{krein-petunin-semenov}), from this inequality  it follows
\begin{align*}
\bigg\|\sum_{n=1}^l y_n\bigg\|_E & 
= \bigg\|\sigma_2\Big(\sigma_{1/2}\Big(\sum_{n=1}^l y_n\Big)\Big)\bigg\|_E
\\ & 
\le 
2  \bigg\|\sigma_{1/2}\Big(\sum_{n=1}^l y_n\Big)\bigg\|_E
\le 
2  \bigg\|\sum_{n=1}^l f_n\bigg\|_E.
\end{align*}
Therefore, combining this together with \eqref{16}, we have
\begin{equation}\label{18}
\bigg\|\sum_{n=1}^l x_n\bigg\|_E\le 2\, \bigg\|\sum_{n=1}^l f_n\bigg\|_{E}.
\end{equation}

On the other hand, from \eqref{15} it follows that there are pairwise disjoint sets $F_n^{k,i}\subseteq[0,1]$ such that $m(F_n^{k,i})=m(A_i)$, $i\in S_n^k$, $k\in\mathbb{N}$, $n=1,\dots,l$, and
$$
F_n^k=\bigcup_{i\in S_n^k} F_n^{k,i},\quad k\in\N,\;n=1,\dots,l.
$$
Moreover, by the above definitions, $v_n^{k,i}$ are being independent copies of the characteristic functions $\chi_{A_i}$, $i\in S_n^k$, $k\in\N$, $n=1,2,\dots,l$. Since the sets $A_i$ (resp. $F_n^{k,i}$), $i\in S_n^k$, $k\in\N$, $n=1,2,\dots,l$,  are pairwise disjoint and 
$$
\sum_{n=1}^l\sum_{k=1}^\infty\sum_{i\in S_n^k} m(A_i)\le \sum_{n=1}^l m(\{t:x_n(t)\ne 0\})\le1,$$
by the hypothesis of the proposition (see also Remark \ref{r1}), we have
\begin{eqnarray}
\bigg\|\sum_{n=1}^l f_n\bigg\|_{E} &\le& 
C \bigg\|\sum_{n=1}^l \sum_{k=1}^\infty\varepsilon k\sum_{i\in S_n^k}\chi_{A_i}\bigg\|_{Z_E}\nonumber\\ &=& 
C \bigg\|\sum_{n=1}^l \sum_{k=1}^\infty \varepsilon k\sum_{i\in S_n^k}\chi_{F_n^{k,i}}\bigg\|_{E}\nonumber\\ &=& 
C\, \bigg\|\sum_{n=1}^l z_n\bigg\|_{E}, 
\label{new2}
\end{eqnarray}
where $z_n:=\sum_{k=1}^\infty \varepsilon k\cdot\chi_{F_n^k}$, $n=1,2,\dots,l$.

%Moreover, since the sets 
%$G_n^{k,i}$, $i\in S_n^k$, $n=1,\dots,m$, $k\in\N$
%(respectively, $F_n^{k}$, $n,k\in\N$) are pairwise disjoint and 
%$$
%\sum_{i\in S^k_n} m(G_n^{k,i})=m(G_n^k)=m(F_n^k),
%\quad n=1,\dots,m, k\in\N,
%$$
%the the functions  
%$$
%\sum_{n=1}^m y_n\quad\text{and}\quad \sum_{n=1}^m z_n,
%\quad\text{where}\quad
%z_n:=\sum_{k=1}^\infty \varepsilon_k\chi_{F_n^k},
%$$
%are equidistributed. Therefore, from \eqref{new2} we deduce the inequality.
%
%$$
%\bigg\|\sum_{n=1}^m f_n\bigg\|_{E} \le 
%C\,\bigg\|\sum_{n=1}^m\sum_{k=1}^\infty \varepsilon k 
%\sum_{i\in S_n^k} \chi_{A'_i}\bigg\|_E,
%$$
%where the sets $A'_i\subseteq [0,1]$ are defined in Remark \ref{r1}. Since
%the r.v.s 
%$$
%\sum_{n=1}^m\sum_{k=1}^\infty\varepsilon k 
%\sum_{i\in S_n^k} \chi_{A'_i}
%\quad\text{and}\quad
%\sum_{n=1}^mz_n=\sum_{n=1}^m\sum_{k=1}^\infty\varepsilon k 
%\sum_{i\in S_n^k} \chi_{F_n^k}
%$$
%are equidistributed, then from the preceding inequality  it follows
%
%
%
%\begin{equation}\label{19}
%\bigg\|\sum_{n=1}^m f_n\bigg\|_E\le C\, \bigg\|\sum_{n=1}^m z_n\bigg\|_{E}.
%\end{equation}
%

Further, for every $n=1,\dots,l$, $k=2,3,\dots$ and all $t\in F_n^k$ we have
$$
\overline{x}_n(t)> \varepsilon (k-1)\ge \frac12\, \varepsilon k= \frac12\, z_n(t).
$$
Hence, taking into account the disjointness of the sets $F_n^k$, $k\in\N$, $n=1,\dots,l$, we obtain
$$
\bigg\|\sum_{n=1}^l \overline{x}_n\bigg\|_E\ge \frac12 
\bigg\|\sum_{n=1}^l z_n\sum_{k=2}^\infty \chi_{F_n^k}\bigg\|_E.
$$
Additionally, since the sets $F_n^1\subseteq[0,1]$, $n=1,2,\dots,l$, are pairwise disjoint, then
$$
\bigg\|\sum_{n=1}^l z_n\chi_{F_n^1}\bigg\|_E
\le\varepsilon\, \|\chi_{0,1]}\|_E=
\varepsilon.
$$
As a result, from  inequalities \eqref{18} and \eqref{new2} we get
\begin{align*}
\bigg\|\sum_{n=1}^l x_n\bigg\|_E &\le 2C\, \bigg\|\sum_{n=1}^l z_n
\bigg\|_E 
\\ &
\le 2C \left(\bigg\|\sum_{n=1}^l z_n\sum_{k=2}^\infty \chi_{F_n^k}\bigg\|_E
+\bigg\|\sum_{n=1}^l z_n\chi_{F_n^1}\bigg\|_E\right)
\\ & \le 
4C \bigg(\varepsilon+\bigg\|\sum_{n=1}^l \overline{x}_n\bigg\|_E\bigg).
\end{align*}
Letting $\varepsilon\to0$, we obtain \eqref{14} with $C'=4C$.
\end{proof}

Next, we proceed with comparing the sequence $\{v_i\}$ with the sequence $\{u_i\}$ of independent symmetrically distributed r.v.'s equimeasurable with the characteristic functions $\chi_{A_i}$, $i=1,2,\dots$.

%%%%%%%%%%%%%%%%%%%%%%%%%%%%%%

\begin{proposition}\label{p1}
Let $E$ be a r.i.\ space on $[0,1]$. Then, for every $S\subseteq\N$ such that $\sum_{i\in S} m(A_i)\le1$ and all $a_i\in\R$, $i\in S$, we have
\begin{equation}\label{7}
\bigg\|\sum_{i\in S} a_iv_i\bigg\|_E\le 16\, e\cdot
\bigg \|\sum_{i\in S} a_i u_i\bigg\|_{E}.
\end{equation}
\end{proposition}

%%%%%%%%%%%%%%%%%%%%%%%%%%%%%%

\begin{proof}
First, since $u_i$, $i=1,2,\dots$, are independent symmetrically distributed r.v.'s, the sequence $\{u_n\}_{n=1}^\infty$ is 1-unconditional in $E$
(see, e.g., \cite[Proposition~1.14]{B}). Therefore, we may (and will) assume that coefficients $a_i$, $i\in S$, are nonnegative. 

For each $i\in S$, recalling that $m(A_i)>0$, we denote by $\alpha_i$ the least root of the equation
$$
2t(1-t)=\frac14 m(A_i).
$$
Straightforward calculations show that 
\begin{equation}\label{8}
\frac18 m(A_i)< \alpha_i< \frac12 m(A_i),\quad i\in S.
\end{equation}

Let $\{G_i, H_i\}_{i\in S}$ be a family of independent subsets of $[0,1]$
such that $m(G_i)=m(H_i)=\alpha_i$, $i\in S$. Then, clearly,
$h_i:=\chi_{H_i}-\chi_{G_i}$, $i\in S$, are independent symmetrically distributed r.v.'s. Moreover, since $m(\{t:|u_i(t)|=1\})=m(A_i)$ for each $i\in S$, and, due to independence, 
$$
m(\{t:|h_i(t)|=1\})=2\alpha_i(1-\alpha_i)=\frac14m(A_i),\;\;i\in S,$$ 
we have
$$
m(\{t:|h_i(t)|>\tau\}) \le m(\{t:|u_i(t)|>\tau\}),\quad \tau>0.
$$
Hence, by the well-known Kwapien-Rychlik inequality (see e.g. \cite[Ch. V, Theorem 4.4]{VTCh}), for all $a_i\ge 0$ and $\tau>0$,   we get
\begin{equation}\label{9}
m\Big(\Big\{t:\Big|\sum_{i\in S} a_ih_i(t)\Big|>\tau\Big\}\Big)
\le 2 
m\Big(\Big\{t:\Big|\sum_{i\in S} a_iu_i(t)\Big|>\tau\Big\}\Big).
\end{equation}

Next, denoting $h:=\sum_{i\in S} a_ih_i$, we represent $h=h'-h''$, where
$$
h':=\sum_{i\in S} a_i\chi_{H_i},\quad h'':=\sum_{i\in S} a_i\chi_{G_i}.
$$
Since $h'$ and $h''$ are independent, for each $\tau>0$ it follows
\begin{align}\label{10}
m(\{t:|h(t)|>\tau\}) & \ge
m(\{t:|h'(t)|>\tau\}\cap \{t:h''(t)=0\}) \nonumber
\\ & =
m(\{t:|h'(t)|>\tau\})\cdot m(\{t:h''(t)=0\}).
\end{align}
Further, since $G_i$ are independent, by \eqref{8}, we have
\begin{align}\label{11}
m(\{t:h''(t)=0\}) 
&\ge 
m\Big(\bigcap_{i\in S}([0,1]\setminus G_i)\Big)=
\prod_{i\in S}(1-m(G_i))\nonumber
\\ & = \prod_{i\in S}(1-\alpha_i)\ge \prod_{i\in S}(1-\frac12m(A_i)). 
\end{align}
Finally,  from the elementary inequality
$$
\log(1-x)\ge -\frac{x}{1-x},\;\;0\le x<1,$$
and the assumption that $\sum_{i\in S} m(A_i)\le1$ it follows
\begin{align*}
\log\bigg(\prod_{i\in S}(1-\frac12m(A_i)) \bigg)
&= 
\sum_{i\in S}\log\Big(1-\frac12m(A_i) \Big) 
\\ & \ge
-\frac12\sum_{i\in S} \frac{m(A_i)}{1-\frac12 m(A_i)}
\\ & \ge 
-\sum_{i\in S} m(A_i) \ge -1.
\end{align*}
Combining the latter inequality with \eqref{10} and \eqref{11}, we obtain
\begin{equation}\label{12}
m\Big(\Big\{t:\Big|\sum_{i\in S} a_ih_i(t)\Big|>\tau\Big\}\Big)
\ge \frac1e \,
m\Big(\Big\{t:\Big|\sum_{i\in S} a_i\chi_{H_i}(t)\Big|>\tau\Big\}\Big).
\end{equation}

On the other hand, one can easy see that, by  \eqref{8}, for all $i\in S$
$$
m(\{t:v_i(t)>\tau\}) \le 8m(\{t:\chi_{H_i}(t)(t)>\tau\}),\quad \tau>0.
$$
%clearly $v_i=\chi_{A'_i}$, where $A_i'$, $i\in S$, are independent
%subsets of $[0,1]$ with $m(A_i')=m(A_i)$, $i\in S$.
%Moreover, by  \eqref{8} $m(A_i')\le 8 m(H_i)$, $i\in S$.
Therefore, by passing to the rearrangements of r.v.'s $v_i$ and $\chi_{H_i}$, $i\in S$, in the same way as in the proof of Proposition \ref{p2}, we deduce that for all $\tau>0$ and $a_i\ge 0$
$$
m\Big(\Big\{t:\Big|\sum_{i\in S} a_i\chi_{H_i}(t)\Big|>\tau\Big\}\Big)
\ge \frac18 \,
m\Big(\Big\{t:\Big|\sum_{i\in S} a_iv_i(t)\Big|>\tau\Big\}\Big).
$$

Summing up this inequality, \eqref{9} and \eqref{12}, we arrive at the estimate
$$
m\Big(\Big\{t:\Big|\sum_{i\in S} a_iv_i(t)\Big|>\tau\Big\}\Big)
\le 16 \, e\cdot 
m\Big(\Big\{t:\Big|\sum_{i\in S} a_iu_i(t)\Big|>\tau\Big\}\Big),
\quad \tau>0.
$$
As a result, applying \cite[Ch.II, \S4.3, Corollary 2]{krein-petunin-semenov}, we obtain \eqref{7}.
\end{proof}

%%%%%%%%%%%%%%%%%%%%%%%%%%%%%%

Now, from Propositions \ref{p2}, \ref{p1}, \cite{AS-08} (or \cite[Theorem 25]{AS-10}), and \cite[Theorem 2.4]{A-10} it follows the first main result of the paper.

\begin{theorem}\label{t1}
Let  $E$ be a r.i.\ space  on $[0,1]$. Suppose there is a constant $C>0$ such that for every set $S\subseteq\N$, with $\sum_{n\in S} m(A_n)\le1$, and all $a_n\in\R$, $n\in S$, we have 
\eqref{4weak}, that is,
\begin{equation*}
\label{4weakweak}
\bigg\|\sum_{n\in S} a_nu_n\bigg\|_E\le  C\, \bigg\|\sum_{n\in S} a_n \chi_{A_n}\bigg\|_{Z_E},
\end{equation*}
where $u_n$ are independent symmetrically distributed functions, equimeasurable with  $\chi_{A_n}$.
Then, the Kruglov operator $K$ is bounded from $E$ into $E''$. 

Therefore, if $E$ has the Fatou property, then it possesses the Kruglov property and hence there is a constant $C>0$, depending only on $E$, such that for every sequence $\{x_n\}_{n=1}^\infty$ of independent symmetrically distributed r.v.'s from $E$ inequality \eqref{4} holds, that is, 
\begin{equation*}
\bigg\|\sum_{n=1}^\infty x_n\bigg\|_E
\le C\bigg\|\sum_{n=1}^\infty \overline{x}_n\bigg\|_{Z_E},
\end{equation*}
where $\{\overline{x}_n\}_{n=1}^\infty$ is a disjointification of $\{x_n\}_{n=1}^\infty$.

If we assume that, additionally, for some constant $C>0$ and  every $S\subseteq\N$ such that $\sum_{n\in S} m(A_n)\le1$ and all $a_n\in\R$, $n\in S$, 
\begin{equation*}
\bigg\|\sum_{n\in S} a_nu_n\bigg\|_{E'}\le  C \bigg\|\sum_{n\in S} a_n \chi_{A_n}\bigg\|_{Z_{E'}},
\end{equation*}
then the spaces $E$ and $Z_{E}$ are isomorphic.
\end{theorem}

%%%%%%%%%%%%%%%%%%%%%%%%%%%%%%

%Hence, applying \cite[Theorem 22(ii)]{AS-10}, we obtain
%
%
%\begin{corollary}\label{c1}
%Suppose that $E$ is an r.i.s.\ on $[0,1]$ such that $E''=E$ and for all
%$a_n\in\R$ inequality \eqref{5} holds. Then (also for all $a_n\in\R$), we have \eqref{6}.
%\end{corollary}

%%%%%%%%%%%%%%%%%%%%%%%%%%%%%%

Theorem \ref{t1} asserts that $E\approx Z_{E}$ under some conditions related to both spaces $E$ and $E'$. Next, we prove a statement, showing that the same result holds provided that, along with inequality 
\eqref{4weak}, %\eqref{4weakweak}, 
the subspace $[u_n]$ of $E$ has a certain geometrical property.

We will repeatedly use the following auxiliary result.

\begin{lemma}\label{l1}
For every r.i.\ space $E$ on $[0,1]$, we have $(Z_E)'=Z_{E'}$. Moreover, if $E$ has the Fatou property (resp. is separable), then so has (resp. is) $Z_E$.
\end{lemma}

%%%%%%%%%%%%%%%%%%%

\begin{proof}
Since $Z_E$ is a r.i.\ space on $[0,\infty)$, then 
$$
\|y\|_{(Z_E)'}=\sup_{\|x\|_{Z_E}\le 1}\int_0^\infty x^*(t)y^*(t)\,dt$$
(see, for instance, \cite[Ch.II, \S2.2,  property  $14^0$]{krein-petunin-semenov}). Hence, by definition of the norm in $Z_E$, we have
\begin{align*}
\|y\|_{(Z_E)'} 
&\asymp \sup_{\|x\|_{E}\le 1}\int_0^1 x^*(t)y^*(t)\,dt+
\sup_{\|(x^*(k))_{k=1}^\infty\|_{l_2}\le 1}\sum_{k=1}^\infty x^*(k)y^*(k)
\\ &= 
\|y^*\chi_{[0,1]}\|_{E'}+\Big(\sum_{k=1}^\infty y^*(k)^2\Big)^{1/2}
\\ & \asymp
\|y\|_{Z_{E'}^2},
\end{align*}
and the first assertion of the lemma follows.

Next, suppose that $E$ has the Fatou property. Let  a sequence $\{x_n\}_{n=1}^\infty\subseteq Z_E$ satisfy the conditions  $0\le x_n\uparrow x$ and $\sup_n\|x_n\|_{Z_E}<\infty$.  Observe that then $x_n^*\uparrow x^*$ a.e.\ on $[0,1]$  
(see e.g. \cite[Ch.II, \S2.2,  property $11^0$]{krein-petunin-semenov}).  
Therefore, by the hypothesis and the inequality
$$
\max\Big\{\sup_n\|x_n^*\chi_{[0,1]}\|_E,
\sup_n\|x_n^*\chi_{[1,\infty)}\|_{L^2}\Big\}\le
\sup_n\|x_n\|_{Z_{E}}<\infty,$$
we have $x^*\chi_{[0,1]}\in E$ and $x^*\chi_{[1,\infty)}\in L^2(0,\infty)$.
As a result, $x\in Z_E$ and $\|x\|_{Z_E}=\lim_{n\to\infty}\|x_n\|_{Z_E}.$ This means that $Z_E$ has the Fatou property.

It remains to prove that $Z_E$ is separable provided if $E$ is. To this end, in view of \cite[Ch.II, \S4.5, Theorem~4.8]{krein-petunin-semenov}, it suffices to show that each nonnegative function $x\in Z_E$ can be approximated in $Z_E$ by its truncations, i.e., we need to deduce that $\|x-x_n\|_{Z_E}\to 0$ and $\|x-x^n\|_{Z_E}\to 0$ as $n\to\infty,$
where $x_n:=x\chi_{[0,n]}$ and $x^n:=\min(x,n)$, $n\in\N$.

Let $\varepsilon>0$ be arbitrary. Since $E$ and $L^2(0,\infty)$ are separable r.i.\ spaces, there is $\delta>0$ such that
\begin{equation}\label{eq100}
\max\Big\{\|x^*\chi_{[0,\delta]}\|_E,
\|x^*\chi_{[1,1+\delta]}\|_{L^2}\Big\}<\varepsilon.
\end{equation}
On the other hand, taking into account that
$m\{t>0:\,x(t)>\varepsilon\}<\infty$ and $\|x^*\chi_{[n,\infty)}\|_{L^2(0,\infty)}\to 0$ as $n\to\infty,$ we can find a positive integer $N$ satisfying the conditions:
\begin{equation}\label{eq101}
m(\{t>N:\,x(t)>\varepsilon\})<\delta
\end{equation}
and
\begin{equation}\label{eq102}
\|x^*\chi_{[N,\infty)}\|_{L^2[0,\infty)}<\varepsilon.
\end{equation}
From definition of the rearrangement of a measurable function and inequality \eqref{eq101} it follows that, for all $n\ge N$,
$$
m(\{t>0:\,(x\chi_{[n,\infty)})^*(t)>\varepsilon\})=m(\{t>n:\,x(t)>\varepsilon\})<\delta.
$$
Combining this inequality with \eqref{eq100}, we have
\begin{align}\label{eq103}
\|(x\chi_{[n,\infty)})^*\chi_{[0,1]}\|_E
& \le \|x^*\chi_{[0,\delta]}\|_E+
\|(x\chi_{[n,\infty)})^*\chi_{[\delta,1]}\|_E\nonumber
\\ & \le 
\varepsilon(1+\|\chi_{[0,1]}\|_E)=2\varepsilon
\end{align}
(because $\|\chi_{[0,1]}\|_E=1$; see Section~\ref{prel1}).
Moreover, since
$$
m(\{t>0:\,x(t)\chi_{[n,\infty)}(t)>x^*(N)\})\to 0\quad\mbox{as}\;\; n\to\infty,
$$
there exists a positive integer $M>N$ such that for all $n\ge M$
$$
m(\{t>0:\,(x\chi_{[n,\infty)})^*(t)>x^*(N)\})=m(\{t>0:\,x(t)\chi_{[n,\infty)}(t)>x^*(N)\})<\delta.$$
Hence, from \eqref{eq100} it follows that
$$
\|(x\chi_{[n,\infty)})^*\chi_{[1,\infty)}\chi_{\{
(x\chi_{[n,\infty)})^*> x^*(N)\}}\|_{L^2}\le \|x^*\chi_{[1,1+\delta]}\|_{L^2}<\varepsilon,\;\;n\ge M.
$$
On the other hand, in view of \eqref{eq102},  
\begin{align*}
\|(x\chi_{[n,\infty)})^*\chi_{\{(x\chi_{[n,\infty)})^*\le x^*(N)\}}\|_{L^2} 
&\le 
\|x^*\chi_{\{x^*\le x^*(N)\}}\|_{L^2} 
\\ & \le 
\|x^*\chi_{[N,\infty)}\|_{L^2}<\varepsilon,\;\;n\ge M.
\end{align*}
Summing up the last inequalities, we have that for all $n\ge M$ 
\begin{align*}
\|(x\chi_{[n,\infty)})^*\chi_{[1,\infty)}\|_{L^2} &\le 
\|(x\chi_{[n,\infty)})^*\chi_{[1,\infty)}\chi_{\{
(x\chi_{[n,\infty)})^*> x^*(N)\}}\|_{L^2}
\\ &+
\|(x\chi_{[n,\infty)})^*\chi_{\{
(x\chi_{[n,\infty)})^*\le x^*(N)\}}\|_{L^2}\le 2\varepsilon.
\end{align*}
This inequality and \eqref{eq103} imply that $\|x\chi_{[n,\infty)}\|_{Z_E}\le 4\varepsilon$ for all $n\ge M$. Since $\varepsilon>0$ is arbitrary, this yields $\|x-x_n\|_{Z_E}\to 0$ as $n\to\infty$.

Finally, we prove a similar assertion for the upper truncations $x^n$, $n\in\N$. 
Suppose that, as above, $\delta>0$ satisfies condition \eqref{eq100}. Then, if a positive integer $N'$ is sufficiently large, we have $m(\{t>0:\,x(t)>N'\})<\delta$. Combining this inequality with \eqref{eq100}, for all $n\ge N'$ we get 
$$
\|x-x^n\|_{Z_E}=\|x\chi_{\{x\ge n\}}\|_{Z_E}\le\|x^*\chi_{[0,\delta]}\|_E<\varepsilon,$$
whence $\|x-x^n\|_{Z_E}\to 0$ as $n\to\infty.$
\end{proof}

%%%%%%%%%%%%%%%%%%%%%%%%%%%%%%

Let $\{A_n\}_{n=1}^\infty$ be a sequence of pairwise disjoint 
measurable subsets of $(0,\infty)$ satisfying conditions \eqref{2}. Moreover, let $E$ be a r.i.\ space on $[0,1]$ and $\phi_E$ its fundamental function.
Denoting by  $u_n$, $n=1,2,\dots$, supported on $[0,1]$ independent symmetrically distributed r.v.'s, which are equimeasurable with the characteristic functions $\chi_{A_n}$, $n=1,2,\dots$, we set
\begin{align}\label{20}
f_n&:=\frac{\chi_{A_n}}{\phi_E(m(A_n))},\quad
g_n:=\frac{\chi_{A_n}}{\phi_{E'}(m(A_n))},\nonumber\\
\tilde{f}_n&:=\frac{u_n}{\phi_E(m(A_n))},\quad
\tilde{g}_n:=\frac{u_n}{\phi_{E'}(m(A_n))},\quad n=1,2,\dots
\end{align}
Since $\phi_{E'}(t)=t/\phi_E(t)$, $0<t\le1$ 
\cite[Ch.II, \S4.6]{krein-petunin-semenov}, then $\{{f}_n,{g}_n\}$ and $\{\tilde{f}_n,\tilde{g}_n\}$ are biorthogonal systems in $E$. Also, we denote
$$
\langle f,g\rangle:=\int_0^1f(t)g(t)\,dt,\quad f\in E, g\in E'.
$$

%%%%%%%%%%%%%%%%%%%%%%%%%%%%%%

\begin{proposition}\label{p3}
Let $E$ be a r.i.\ space on $[0,1]$, and let $S\subseteq\N$ be such that $\sum_{i\in S}m(A_i)\le1$. Suppose that the mapping
\begin{equation}\label{21}
Pf:=\sum_{n\in S} \langle f,\tilde{g}_n\rangle \tilde{f}_n
\end{equation}
is a bounded projection on $E$. Then, there is a constant $C>0$, which depends only on $E$ and $\|P\|$, such that for all $a_n\in\R$
\begin{equation}\label{22}
\bigg\|\sum_{n\in S} a_nu_n\bigg\|_{E'}\le  C\, \bigg\|\sum_{n\in S} 
a_n \chi_{A_n}\bigg\|_{Z_{E'}}.
\end{equation}
\end{proposition}

%%%%%%%%%%%%%%%%%%%%%%%%%%%%%%

\begin{proof}
First, we estimate
\begin{align*}
\bigg\|\sum_{n\in S} a_n\tilde{g}_n\bigg\|_{E'} & = 
\sup\Big\{\Big\langle \sum_{n\in S} a_n\tilde{g}_n, f\Big\rangle:\, \|f\|_E\le1\Big\}
\\ & =
\sup\Big\{\Big\langle \sum_{n\in S} a_n\tilde{g}_n, Pf\Big\rangle:\, \|f\|_E\le1\Big\}
\\ &\le
\sup\Big\{\Big\langle \sum_{n\in S} a_n\tilde{g}_n, Pf\Big\rangle:\, 
\|Pf\|_E\le\|P\|\Big\}.
\end{align*}
Moreover,
$$
\Big\langle \sum_{n\in S} a_n\tilde{g}_n, Pf\Big\rangle=\sum_{n\in S} a_n\big\langle  f,\tilde{g}_n\big\rangle =
\int_0^\infty \Big(\sum_{n\in S} a_n g_n\Big)\cdot
 \Big(\sum_{m\in S} \big\langle  f,\tilde{g}_m\big\rangle f_m\Big)\,dt,
$$
 and since $f_m$ are disjoint copies of the functions $\tilde{f}_m$, $m\in S$, by \cite[Theorem 1]{JS}, there is $C'>0$, depending only on $E$, such that 
$$
\bigg\|\sum_{m\in S} \big\langle  f,\tilde{g}_m\big\rangle f_m\bigg\|_{Z_E}
\le C' \,
\bigg\|\sum_{m\in S} 
\big\langle  f,\tilde{g}_m\big\rangle \tilde f_m\bigg\|_{E}
= C' \, \|Pf\|_E.
$$
Hence,
\begin{align*}
\bigg\|\sum_{n\in S} a_n\tilde{g}_n\bigg\|_{E'} 
\le &
\sup\bigg\{
\int_0^\infty \Big(\sum_{n\in S} a_n g_n\Big)\cdot
 \Big(\sum_{m\in S} \big\langle  f,\tilde{g}_m\big\rangle f_m\Big)\,dt
\\ &: \bigg\|\sum_{m\in S} \big\langle  f,\tilde{g}_m\big\rangle f_m\bigg\|_{Z_E}
\le C'\, \|P\|\bigg\}.
\end{align*}
Since $(Z_E)'=Z_{E'}$, by Lemma \ref{l1}, the latter inequality yields that for all $a_n\in\R$ we obtain the inequality
$$
\bigg\|\sum_{n\in S} a_n\tilde{g}_n\bigg\|_{E'} 
\le C'\,\|P\| 
\bigg\|\sum_{n\in S} a_n g_n\bigg\|_{Z_{E'}}, 
$$
which is equivalent to desired estimate \eqref{22}.
\end{proof}

%%%%%%%%%%%%%%%%%%%%%%%%%%%%%%

From Theorem \ref{t1} and Proposition~\ref{p3} it follows

\begin{theorem}\label{t2}
Let  $E$ be a r.i.\ space  on $[0,1]$ with the Fatou property. Suppose that 
there exists $C>0$ such that for every set $S\subseteq\N$,
with $\sum_{n\in S} m(A_n)\le1$, and all  $a_n\in\R$, $n\in S$, inequality 
\eqref{4weak} %\eqref{4weakweak} 
holds and
%
%\begin{equation*}
%\bigg\|\sum_{n\in S} a_nu_n\bigg\|_E\le  C\, \bigg \|\sum_{n\in S} a_n %\chi_{A_n}\bigg\|_{Z_E}.
%\end{equation*}
%
the projection $P$ corresponding to such a set $S\subseteq\N$ (see \eqref{20} and \eqref{21}) is bounded on $E$. Then, $E\approx Z_{E}$.
\end{theorem}

%%%%%%%%%%%%%%%%%%%%%%%%%%%%%%%%%%%%%

\section{Existence of an isomorphic embedding $T\colon {\mathcal{U}}_E\to E$: the case when $T(\chi_{A_n})$, $n=1,2,\dots$, are "almost"\:disjoint.}

%%%%%%%%%%%%%%%%%%%%%%%%%%%%%%%%%%%%%

As was said in Section~\ref{Intro}, if a r.i.\ space $E$ and its K\"othe dual $E'$ possess the Kruglov property, then the spaces $E$ and $Z_E$ are isomorphic  (see \cite{A-10}). In turn, according to Theorem~\ref{t1}, a r.i.\ space $E$ with the Fatou property has the Kruglov property whenever there is an isomorphic embedding of Rosenthal's space ${\mathcal{U}}_E$ into $E$. Moreover, in the proof of the latter result the functions $T(\chi_{A_n})\,(=u_n)$, $n=1,2,\dots$, were independent, symmetrically distributed and equimeasurable with the characteristic functions $\chi_{A_n}$, $n=1,2,\dots$. A natural question appears: Let $T$ be an isomorphic embedding of Rosenthal's space ${\mathcal{U}}_E$ into $E$. What we can say about the functions $T(\chi_{A_n})$, $n=1,2,\dots$? Further, we consider two different cases, when these functions are "almost"\:disjoint and independent. As a consequence, we will obtain new examples of r.i.\ spaces $E$ such that $E\not\approx Z_E$.

We begin with an auxiliary result, which was proved earlier in the separable case by Raynaud (see \cite[Proposition~1]{raynaud}). However, for the reader's convenience we provide here a simple alternative proof of this fact.
Let $G$ denote the separable part of the exponential Orlicz space $\text{Exp}L^2$ (i.e., the closure of $L^\infty$ in $\text{Exp}L^2$).
%$L_{M_2}$, for $M_2(u)=e^{u^2}-1$.

%%%%%%%%%%%%%%%%%%%%%%%%%%%%%%

\begin{proposition}\label{l2}
Let $E$ be a r.i.\ space on $[0,1]$. Suppose that there exists a 
sequence $\{x_n\}_{n=1}^\infty\subseteq E$ with $\|x_n\|_E\asymp \|x_n\|_{L^1}$, $n=1,2,\dots$, which is equivalent in $E$ to the unit vector $\ell^2$-basis.
Then, $E\supset G$.
\end{proposition}

%%%%%%%%%%%%%%%%%%%%%%%%%%%%%%

\begin{proof}
Clearly, it can be assumed that $E\not=L^1$. Since $\{x_n\}_{n=1}^\infty$ is equivalent  in $E$
to the unit $\ell^2$-basis, we have $x_n\to0$ weakly in $E$ and so $x_n\to0$ weakly
in $L^1$. Hence, $\{x_n\}_{n=1}^\infty$ has no convergent subsequences in
$L^1$. Applying
then the  well-known result by Aldous and Fremlin \cite{aldous-fremlin}, 
we select a subsequence $\{x_{n_k}\}\subseteq \{x_n\}$ such that for some $c>0$ and all $a_k\in\R$
$$
\bigg\|\sum_{k=1}^\infty a_kx_{n_k}\bigg\|_{L^1}\ge c\, \|(a_k)\|_{2}.
$$
Combining this inequality with the assumptions and with the 
embedding $E\subseteq L^1$, we conclude that the norms of $E$ and $L^1$ are equivalent on the infinite-dimensional subspace $[x_{n_k}]$ in $E$.

In other words, the canonical embedding $I\colon E\to L^1$ is not strictly singular.
Assuming that  $E\not\supseteq G$, by  \cite[Theorem 2]{AHS}, we obtain 
that this embedding is not disjointly strictly singular. This means that there 
is a sequence of pairwise disjoint functions $\{h_i\}_{i=1}^\infty$ from $E$ such that 
the norms of $E$ and $L^1$ are equivalent on the closed linear span $[h_i]$. 
But this is a contradiction. Indeed, if the norms of $E$ and $L^1$ were equivalent on the span $[h_i]$ of pairwise disjoint functions
$h_i$, $i=1,2,\dots$, one can easily check that there exists $\delta>0$ such that for every 
$i=1,2,\dots$
$$
m(\{t\in[0,1]:|h_i(t)|>\delta\|h_i\|_E\})>\delta
$$
(see also \cite[Theorem 1]{KP}). Clearly, the sets 
$$
U_i(\delta):=\{t\in[0,1]:|h_i(t)|>\delta\|h_i\|_E\},\quad i=1,2,\dots,
$$
are pairwise disjoint and $m({U}_i(\delta))>\delta$. Hence,
$$
m\Big(\bigcup_{i=1}^\infty {U}_i(\delta)\Big)=
\sum_{i=1}^\infty m({U}_i(\delta))=\infty,
$$
which is not possible because the union $\bigcup_{i=1}^\infty {U}_i(\delta)$ is contained in $[0,1]$ (other proofs of this and some close results see in  \cite{N} and \cite[Corollary 3]{A-99}).
\end{proof}

\begin{corollary}
\label{cor1}
Suppose $E$ is a separable r.i.\ space on $[0,1]$ such that $E\not\supset G$. Then, if $E$ contains a sequence $\{x_n\}_{n=1}^\infty$ equivalent in $E$ to the unit vector $\ell^2$-basis, then there is a disjoint sequence $\{x_n\}_{n=1}^\infty\subset E$ with the same property.
\end{corollary}
\begin{proof}
By Proposition~\ref{l2}, we may assume that $\|x_n\|_{E}/\|x_n\|_{L^1}\to \infty$ as $n\to\infty$. Then, by the Kadec-Pe\l czy\'nski alternative \cite{KP}, 
there is subsequence $\{x_{n_j}\}\subseteq \{x_{n}\}$ such that for some disjoint sequence $\{z_j\}\subseteq E$ we have
$$
\|x_{n_j}-z_j\|_E\to0\quad\mbox{as}\;\; j\to\infty.
$$ 
Since $\{x_{n_j}\}$ is equivalent in $E$ to the unit vector $\ell^2$-basis,  applying now the principle of small perturbations (see e.g. \cite[Theorem 1.3.9]{albiac-kalton}), we can assume that $\{z_j\}_{j=1}^\infty$ is equivalent in $E$ to the $\ell^2$-basis as well.
\end{proof}

It is clear that for every r.i.\ space $E$ on $[0,1]$ Rosenthal's space ${\mathcal{U}}_E$ (as a subspace of $Z_E$) contains a subspace isomorphic to $\ell^2$. Hence, if 
${\mathcal{U}}_E\subsetsim E$, the space $E$ must share the above property. So, if a r.i.\  space $E$ does not contain a subspace isomorphic to $\ell^2$, ${\mathcal{U}}_E$ cannot be embedded isomorphically into $E$, which implies that $E\not\approx Z_E$. So, if $E$ is a separable r.i.\ space such that $E\not\supset G$ and it does not contain disjoint sequences equivalent to the unit vector basis of $\ell^2$, then ${\mathcal{U}}_E\not\subsetsim E$ (see Corollary~\ref{cor1}). In particular, if $p>2$, the separable part $(\text{Exp}L^p)_0$ of the exponential Orlicz space $\text{Exp}L^p$ has the latter properties since each disjoint sequence of this space contains a subsequence equivalent to the unit vector basis of $c_0$ (see, e.g., \cite{T}). As a result, we obtain the simplest examples of r.i.\ spaces $E$ such that $E\not\approx Z_E$. 

Further, it is known that, if a r.i.\ space $E$ is not equal to $L^\infty(0,1)$  up to an equivalent renorming, Rosenthal's space ${\mathcal{U}}_E$ contains a {\it complemented} subspace of $Z_E$ isomorphic to $\ell^2$ \cite[Lemma~8.7 and subsequent Remark]{JMST}. Therefore, if we know 
%more than $U_E\subsetsim E$, namely, 
that $E\approx Z_E$, then $E$ must contain a complemented subspace, which is isomorphic to $\ell^2$ as well. According to \cite[Proposition~8.17]{JMST}, there are some  Orlicz spaces,``close'' to $L^1$, that fail to contain such subspaces and hence that are not isomorphic to $Z_E$ (in fact, they are not isomorphic to any r.i.\ space on $(0,\infty)$; see \cite[Corollary~8.15]{JMST}). The simplest example of such a space is the Orlicz space $L_{F_\alpha}$, where $F_\alpha(u)$ is an Orlicz function equivalent to the function $u\log^\alpha u$ for large $u>0$, where $0<\alpha<1/2$ (see also a  discussion in the concluding part of Section~\ref{Intro}). 

Here, we prove results showing that the existence of complemented subspaces isomorphic to $\ell^2$ does not guarantee that ${\mathcal{U}}_E$ is isomorphically embedded into $E$ and, a fortiori, that $E\approx Z_E$. Specifically, we will provide examples of Lorentz spaces containing plenty of complemented subspaces isomorphic to $\ell^2$, but without subspaces isomorphic to the corresponding Rosenthal's spaces. 

%%%%%%%%%%%%%%%%%%%%%%%%%%%%%%

First, we introduce a lattice version of a notion from \cite[see p. 293]{rosenthal}. We say that a Banach
lattice $E$ has the {\it disjoint $Q_2$-property} (in brief, $E\in\mathcal{D}Q_2$) whenever there is a constant $C_E>0$ (depending only on $E$) such that
given a disjoint sequence $\{h_n\}$ in $E$ with $\|h_n\|_E=1$, 
which is equivalent to the unit vector $\ell^2$-basis,
there exists a subsequence $\{h_{n_i}\}\subseteq \{h_n\}$ that is $C_E$-equivalent to the unit vector $\ell^2$-basis.

%\begin{remark}\label{r2}
Let a Banach lattice $E$ have the $\mathcal{D}Q_2$-property (with the constant $C_E$). Suppose that $\{x_n\}_{n=1}^\infty\subset E$ is a disjoint sequence, which is equivalent to the unit $\ell^2$-basis and semi-normalized (i.e., $C^{-1}\le \|x_n\|_E\le C$ for some $C>0$ and all $n=1,2,\dots$). 
Then, it is obvious that $\{x_n\}_{n=1}^\infty$ contains a subsequence, which is $C_E'$-equivalent to the unit vector $\ell^2$-basis, where $C_E':=C_E\cdot C$.
%\end{remark}

%%%%%%%%%%%%%%%%%%%%%%%%%%%%%%%%

\begin{theorem}\label{t3}
Let  $E$ be a separable r.i.\ space, $E\in\mathcal{D}Q_2$. If  ${\mathcal{U}}_E\subsetsim
E$, then $E\supseteq G$.
\end{theorem}

%%%%%%%%%%%%%%%%%%%%%%%%%%%%%%

\begin{proof}
Let $\{A_n\}_{n=1}^\infty$ be a sequence of disjoint subsets of $(0,\infty)$ satisfying conditions \eqref{2}. Then, for every $l\in\N$, there are pairwise
disjoint sets $S^l_i\subseteq\N$, $i=1,2,\dots$, such that
$$
\sum_{n\in S^l_i}m(A_n)=\frac1l.
$$
Denote $B_i^l:=\bigcup_{n\in S_i^l}A_n$, $i=1,2,\dots$.
Consider the block-basis $\{\chi_{B_i^l}\}_{i=1}^\infty$ of
$\{\chi_{A_n}\}_{n=1}^\infty$. According to definition of the norm in $Z_E$ (see \eqref{3}), each set consisting of $l$ distinct functions 
$\chi_{B_i^l}$ is isometrically equivalent in $Z_E$ to the set
$\{\chi_{((i-1)/l,\ i/l)}\}_{i=1}^l$ in $E$, i.e., for all distinct 
$i_1,\dots,i_l\in\N$ and $a_j\in\R$
\begin{equation}\label{23}
\bigg\|\sum_{j=1}^l a_j\chi_{B_{i_j}^l}\bigg\|_{Z_E}
=
\bigg\|\sum_{i=1}^l a_i\chi_{((i-1)/l,\ i/l)}\bigg\|_{E}
\end{equation}
(cf. \cite[Corollary~8]{rosenthal}).

On the other hand, the sequence $\{\chi_{B_i^l}\}_{i=1}^\infty$ is
$C_l$-equivalent in $Z_E$ to the unit vector $\ell^2$-basis. Indeed,
for arbitrary $a_i\in\R$ there is a set $S'_l\subseteq\N$ with
$\text{card }S'_l=l$, such that, with constants depending of $l$, we have
\begin{align}\label{24}
\bigg\|\sum_{i=1}^\infty a_i\chi_{B_{i}^l}\bigg\|_{Z_E}
&=
\bigg\|\sum_{i\in S'_l} a_i\chi_{B_{i}^l}\bigg\|_{E}
+ \bigg\|\sum_{i\not\in S'_l} a_i\chi_{B_{i}^l}\bigg\|_{L^2} \nonumber
\\ & 
\mathrel{\substack{\textstyle C_l \\ \textstyle\asymp}}
\bigg\|\sum_{i\in S'_l} a_i\chi_{B_{i}^l}\bigg\|_{L^2}
+ \bigg\|\sum_{i\not\in S'_l} a_i\chi_{B_{i}^l}\bigg\|_{L^2} \nonumber
\\ & 
\mathrel{\substack{\textstyle 2C_l \\ \textstyle\asymp}}
\bigg\|\sum_{i=1}^\infty a_i\chi_{((i-1)/l,\ i/l)}\bigg\|_{L^2} 
=\frac{1}{\sqrt l}\|(a_i)\|_{2}.
%\\ & =
%\frac{1}{\sqrt l} \bigg(\sum_{i=1}^\infty a_i^2\bigg)^{1/2}.
\end{align}

From the hypothesis, there exists an isomorphism $T\colon {\mathcal{U}}_E\to E$.
Then, if $y_i^l:=T(\chi_{B_i^l})$, $i=1,2,\dots$, by \eqref{24},
we have
\begin{equation}\label{25}
\bigg\|\sum_{i=1}^\infty a_iy_i^l\bigg\|_{E}
\mathrel{\substack{\textstyle \|T\| \\ \textstyle\asymp}}
\bigg\|\sum_{i=1}^\infty a_i\chi_{B_{i}^l}\bigg\|_{Z_E}
\asymp\frac{1}{\sqrt l}\|(a_i)\|_{2},
\end{equation}
with constants depending on $l$ and $\|T\|$.

In the case when $\|y_i^l\|_E\asymp\|y_i^l\|_{L^1}$, $i=1,2,\dots$, for some $l\in\N$, all the conditions of Proposition \ref{l2} are satisfied, and so the 
desired result follows. 

Assume, conversely, that for each $l\in\N$ we have
$$
\liminf_{i\to\infty} \;\frac{\|y_i^l\|_{L^1}}{\|y_i^l\|_E}=0.
$$
Denoting $u_i^l:=(1/\phi_E(1/l))y_i^l$, $i,l=1,2,\dots$, where $\phi_E$ is the fundamental function of the space $E$, we get
\begin{equation}\label{new3}
\|T\|^{-1} \le \|u_i^l\|_E\le \|T\|,\quad i,l=1,2,\dots,
\end{equation}
and clearly for every $l=1,2,\dots$
$$
\liminf_{i\to\infty} \;\frac{\|u_i^l\|_{L^1}}{\|u_i^l\|_E}=0.
$$
Then again, by the Kadec-Pe\l czy\'nski alternative \cite{KP}, 
for each $l=1,2,\dots$ there is  
subsequence $\{u_{i_j}^l\}\subseteq \{u_{i}^l\}$, where a sequence $\{i_j\}$ 
depends on $l\in\N$, such that for some disjoint sequence $\{z_j^l\}\subseteq E$ it holds
$$
\|u_{i_j}^l-z_j^l\|_E\to0\quad\mbox{as}\;\; j\to\infty.
$$
Applying the principle of small perturbations 
(see e.g. \cite[Theorem 1.3.9]{albiac-kalton}), we can assume that $\{z_j^l\}_{j=1}^\infty$ 
is 2-equivalent in $E$ to the sequence $\{u_{i_j}^l\}_{j=1}^\infty$, and so, by \eqref{new3},
$$
(2\|T\|)^{-1} \le \|z_j^l\|_E\le 2\|T\|,\quad j,l=1,2,\dots,
$$
which means that for every $l=1,2,\dots$ the sequence $\{z_j^l\}_{j=1}^\infty$ is semi-normalized with a constant independent of $l$. Moreover, taking into account \eqref{25}, we see that $\{z_j^l\}_{j=1}^\infty$ is equivalent in $E$ to the  unit vector $\ell^2$-basis  (with constants depending on $l=1,2,\dots$). 
Since $E\in\mathcal{D}Q_2$, for each $l\in\N$ the sequence 
$\{z_j^l\}_{j=1}^\infty$ contains a further subsequence $\{z_{j_k}^l\}_{k=1}^\infty$ (where $\{j_k\}$ also depends on $l\in\N$) that is
$C'_E$-equivalent to the unit vector $\ell^2$-basis. 
Clearly, then the sequence $\{u^l_{i_{j_k}}\}_{k=1}^\infty$ is 
$2C'_E$-equivalent to the same basis, i.e.,
\begin{equation}\label{26}
\bigg\|\sum_{k=1}^\infty a_ku_{i_{j_k}}^l\bigg\|_{E}
\mathrel{\substack{\textstyle 2C'_E \\ \textstyle\asymp}}\|(a_k)\|_{2}.
\end{equation}
Moreover, from \eqref{23} and the above notation it follows that
$$
\bigg\|\sum_{k=1}^l a_ku_{i_{j_k}}^l\bigg\|_{E}
\mathrel{\substack{\textstyle \|T\| \\ \textstyle\asymp}}
\frac{1}{\phi_E(1/l)}\bigg\|\sum_{k=1}^l a_k\chi_{B_{i_{j_k}}^l}\bigg\|_{Z_E}
=
\frac{1}{\phi_E(1/l)}\bigg\|\sum_{j=1}^l a_j\chi_{((j-1)/l,\ j/l)}\bigg\|_E
$$
for all $a_j\in\R$. Combining this with \eqref{26}, we obtain
\begin{equation}\label{new4}
\bigg\|\sum_{j=1}^l a_j\chi_{((j-1)/l,\ j/l)}\bigg\|_E
\asymp
\phi_E(1/l)\bigg(\sum_{j=1}^la_j^2\bigg)^{1/2},\quad l\in\N,
\end{equation}
with constants independent of $l\in\N$ and $a_j\in\R$. 

Next, one can easily check that equivalence \eqref{new4} implies that $\phi_E(t)\asymp t^{1/2}$, $0<t\le1$. Indeed, for every $l\in\N$ we have 
$$
\chi_{(0,1)}=\sum_{i=1}^l \chi_{(i-1)/l,i/l)},
$$
whence, by \eqref{new4},
\begin{equation}\label{new5}
1=\big\|\chi_{(0,1)}\big\|_E\asymp \sqrt l \phi_E(1/l).
\end{equation}
Therefore,
%Similarly, for all $m,k\in\N$
%$$
%\chi_{(0,1/m)}=\sum_{i=1}^k \chi_{(i-1)/(km),i/(km))},
%$$
%and therefore 
%$$
%\varphi_E(1/m)=\big\|\chi_{(0,1/m)}\big\|_E\asymp \sqrt k\varphi_E(1/(km)).
%%$$
%In particular, for $m=1$, we get 
$\phi_E(1/l)\asymp 1/\sqrt{l}$, $l\in\N$. Combining 
this together with the quasi-concavity of $\phi_E$, we obtain that
%$$
%\varphi_E(k/m)\asymp \sqrt{k/m},\quad m,k\in\N, k\le m.
%$$
%Since the fundamental function $\varphi_E$ is continuous, we can extend the latter equivalence
%to the whole interval $(0,1)$. Thus, 
$\phi_E(t)\asymp \sqrt{t}$, $0<t\le1$. As a consequence,
from \eqref{new4} it follows that
\begin{align*}
\Big\|\sum_{j=1}^l a_j\chi_{((j-1)/l,j/l)}\Big\|_E
& 
\asymp \frac{1}{\sqrt l} \Big(\sum_{j=1}^l a_j^2\Big)^{1/2}
\\ & =
\Big\|\sum_{j=1}^l a_j\chi_{((j-1)/l,j/l)}\Big\|_{L^2},\quad l\in\N,
\end{align*}
with constants independent of $l\in\N$ and $a_j\in\R$. Clearly, this implies
that $E\approx L^2$, and the desired result follows.
\end{proof}

%%%%%%%%%%%%%%%%%%%%%%%%%%%%%%

\begin{theorem}\label{t4}
Let  $E$ be a separable  r.i.\ space on $[0,1]$ such that both $E$ and
$E'$ have the $\mathcal{D}Q_2$-property. If $E\approx Z_E$,
then $G\subseteq E\subseteq G'$.
\end{theorem}

%%%%%%%%%%%%%%%%%%%%%%%%%%%%%%

\begin{proof}
It follows from Theorem \ref{3} that we need only to prove that $E\subseteq G'$.

Suppose that $T$ is an isomorphism from $Z_E$ onto $E$. 
Clearly, then $T^*$ is an isomorphism from $E^*$ onto $(Z_E)^*$.
Since $E$ is separable, we have $E^*=E'$ and, by Lemma \ref{l1}, 
$Z_E$ is 
a separable space with $(Z_E)^*=(Z_E)'=Z_{E'}$. Thus, $E'\approx Z_{E'}$,
and hence, by Theorem \ref{t3}, $E'\supseteq G$, which implies 
$E\subseteq E''\subseteq G'$.
\end{proof}

%%%%%%%%%%%%%%%%%%%%%%%%%%%%%%

Let $1\le p\le \infty$. Recall that a Banach lattice $E$ is said  to be {\it $p$-disjointly homogeneous}  ($p$-$\mathcal{DH}$) if every disjoint normalized sequence contains a subsequence equivalent to the unit vector $\ell^p$-basis ($c_0$-basis if $p=\infty$). Moreover, $E$ is called {\it uniformly $p$-$\mathcal{DH}$} if there is a constant $C_E$, which depends only on $E$, such that from every disjoint normalized sequence $\{x_n\}$ we can select a subsequence $\{x_{n_k}\}\subseteq\{x_n\}$, which is $C_E$-equivalent to the $\ell^p$-basis (for a detailed account of these properties see the survey \cite {FHT-survey} and references therein).

Every $p$-$\mathcal{DH}$ Banach lattice for $1<p<\infty$ is reflexive \cite{AB}. Also, it is obvious that each uniformly $2$-$\mathcal{DH}$ lattice has  the $\mathcal{D}Q_2$-property. 
%Hence, we get the following consequence of Theorems
%\ref{t3} and  \ref{t4}.

\begin{theorem}\label{t4a}
Let  $E$ be a uniformly $2$-$\mathcal{DH}$ r.i.\ space on $[0,1]$. Suppose that at least one of the following conditions holds:

(i) Rosenthal's space ${{\mathcal{U}}}_{E}$ is isomorphically embedded into the space $E$;

(ii) $E$ is isomorphic to a r.i.\ space on $(0,\infty)$.

Then, $E\supseteq G$.

Moreover, if additionally the K{\" o}the dual $E'$ is uniformly $2$-$\mathcal{DH}$ and $E'$ satisfies at least one of the conditions (i) and (ii), then $G\subseteq E\subseteq G'$.
\end{theorem}
\begin{proof}
Since $E$ is a uniformly $2$-$\mathcal{DH}$, then the condition (i) implies the embedding $E\supseteq G$ by Theorem \ref{t3}.

Let now $E$ be isomorphic to a r.i.\ space $Y$ on $(0,\infty)$. Denote $x_{n,i}:=\chi_{[(i-1)/n,i/n)}$, $n,i\in\N$, and  
%the condition (ii) be fulfilled. If $E\approx Z_{E}$, then clearly we have (i), and hence everything is done.
%Otherwise, suppose that $E$ is isomorphic to a r.i.\ space $Y$  on $(0,\infty)$ such that $Y\not\approx Z_{E}$. 
assume first that, for every $n\in\N$, the sequence $\{x_{n,i}\}_{i=1}^\infty$ is equivalent in $Y$ to the unit vector $\ell^2$-basis. Then, if $T$ is an isomorphism of $Y$ onto $E$, each sequence $\{y_{n,i}\}_{i=1}^\infty$, $n\in\N$, where $y_{n,i}:=T(x_{n,i})$, $n,i\in\N$, is equivalent in $E$ to the unit vector $\ell^2$-basis as well. In the case when $\|y_{n,i}\|_E\asymp\|y_{n,i}\|_{L^1}$, $i=1,2,\dots$, for some $n\in\N$, the desired result follows, as above, by Proposition \ref{l2}. Hence, it remains to consider the case when for each $n\in\N$ we have
$$
\liminf_{i\to\infty} \;\frac{\|y_{n,i}\|_{L^1}}{\|y_{n,i}\|_E}=0.
$$

Then, denoting $u_{n,i}:=(1/\phi_E(1/n))y_{n,i}$, $i,n=1,2,\dots$ and reasoning as in the proof of Theorem \ref{t4}, we can find, for every $n\in\N$, a subsequence $\{u_{n,i_j}\}_{j=1}^\infty$, which is $2$-equivalent in $E$ to some disjoint semi-normalized (with a constant independent of $n$) sequence $\{z_{n,j}\}_{j=1}^\infty$. Thanks to the uniform $2$-$\mathcal{DH}$ property of $E$, passing if it necessary to a further  
subsequence, we can assume that there is a constant $D'>0$ such that for every $n\in\N$ the sequence $\{u_{n,i_j}\}_{j=1}^\infty$ is $D'$-equivalent in $Y$ to the unit vector $\ell^2$-basis. On the other hand, for every $n\in\N$ the sequence $\{y_{n,i}\}_{i=1}^\infty$ (together with $\{x_{n,i}\}_{i=1}^\infty$ in $Y$) is $B$-symmetric in $E$ for some $B>0$. Hence, for every $n\in\N$ the sequence $\{u_{n,i}\}_{i=1}^\infty$ and hence the sequence $\{(1/\phi_E(1/n))x_{n,i}\}_{i=1}^\infty$ is $D$-equivalent in $Y$ to the unit vector $\ell^2$-basis for some $D>0$, i.e., 
$$
D^{-1}\phi_E(1/n)\|(a_i)\|_{2}\le \Big\|\sum_{i=1}^\infty a_ix_{n,i}\Big\|_Y\le D\phi_E(1/n)\|(a_i)\|_{2}$$
for all $n\in\N$ and $(a_i)\in {\ell^2}$. Clearly, this implies that $Y=L^2(0,\infty)$ (see the concluding part of the proof of Theorem \ref{t4}). Since $E\approx Y$ by condition, we infer that $E=L^2[0,1]$ (with equivalence of norms), and so in this case everything is done. 

Conversely, suppose that the sequence $\{y_{1,i}\}_{i=1}^\infty$ is not equivalent in $Y$ to the unit vector $\ell^2$-basis; then, the same is true also for all sequences $\{y_{n,i}\}_{i=1}^\infty$, $n\in\N$. As was said above, for every $n\in\N$ the sequence $\{y_{n,i}\}_{i=1}^\infty$ is $B$-symmetric in $E$ for some $B>0$. Moreover, since $\{x_{n,i}\}_{i=1}^\infty$, $n\in\N$, spans an $1$-complemented subspace in $Y$ (see e.g. \cite[Ch.~II, \S\,3.2]{krein-petunin-semenov}), we can assume that, for every $n\in\N$, the span $[y_{n,i},i\in\N]$ is a $B$-complemented subspace in $E$. Then, according to \cite[Lemma~8.10]{JMST}, there is a constant $A'>0$ such that for every $n\in\N$ the sequence $\{y_{n,i}\}_{i=1}^\infty$ is $A'$-equivalent in $E$ to a disjoint sequence in $E$. Since the latter space is uniformly $2$-$\mathcal{DH}$ and $\{x_{n,i}\}_{i=1}^\infty$ is a $B$-symmetric sequence in $E$, we conclude that there is a constant $A>0$ such that for every $n\in\N$ the sequence $\{(1/\phi_E(1/n))x_{n,i}\}_{i=1}^\infty$ is $A$-equivalent in $Y$ to the unit vector $\ell^2$-basis. As above, this yields that $Y=L^2(0,\infty)$ and hence $E=L^2[0,1]$ (with equivalence of norms), which completes the proof.
\end{proof}

%Let $\varphi$ be an increasing concave function $[0,1]$ with $\varphi(0)=0$. It is well known that every Lorentz space $\Lambda_2(\varphi)$ has the uniform $2$-$\mathcal{DH}$ property (see e.g. \cite[Theorem 5.1]{FJT}). Observe also that $\Lambda_2(\varphi)\subset L_2\subset G'=\Lambda(u\log^{1/2}(e/u))$. Moreover, since each disjoint sequence in $\Lambda_2(\varphi)$ contains a subsequence which spans a $2$-complemented subspace in $\Lambda_2(\varphi)$ \cite[Theorem 5.1]{FJT}, the K{\" o}the dual $(\Lambda_2(\varphi))'$ has the uniform $2$-$\mathcal{DH}$ property as well. 
%Clearly, $(\Lambda_2(\varphi))'\supset G$ and $(\Lambda_2(\varphi))''=\Lambda_2(\varphi)$.  
%Therefore, taking into account that the embedding $\Lambda_2(\varphi)\supseteq G$ is equivalent to the condition $\sum_{k=1}^\infty \varphi(e^{-k})<\infty$ (see e.g. \cite[Lemma~3]{A-20}) and applying Corollary \ref{c3}, we obtain 

It is well known that every Lorentz space $\Lambda_2(\varphi)$ has the uniform $2$-$\mathcal{DH}$ property (see e.g. \cite[Theorem 5.1]{FJT}). Therefore, since the embedding $\Lambda_2(\varphi)\supseteq G$ is equivalent to the condition $\sum_{k=1}^\infty \varphi(e^{-k})<\infty$ (see e.g. \cite[Lemma~3]{A-20}), we get the following consequence of Theorem~\ref{t4a}.

\begin{corollary}\label{c2}
Let $\varphi$ be an increasing concave function on $[0,1]$ with $\varphi(0)=0$. Suppose that at least one of the following conditions holds:

(i) Rosenthal's space ${{\mathcal{U}}}_{\Lambda_2(\varphi)}$ is isomorphically embedded into the space $\Lambda_2(\varphi)$;

(ii) the space $\Lambda_2(\varphi)$ isomorphic to a r.i.\ space on $(0,\infty)$.

Then, $\sum_{k=1}^\infty \varphi(e^{-k})<\infty$.
\end{corollary}

%\begin{proof}
%It is well known that every Lorentz space $\Lambda_2(\varphi)$ has the uniform $2$-$\mathcal{DH}$ property (see e.g. \cite[Theorem 5.1]{FJT}). Therefore, by Corollary \ref{c3}, from (i) it follows that $\Lambda_2(\varphi)\supseteq G$.  Since this embedding is equivalent to the condition $\sum_{k=1}^\infty \varphi(e^{-k})<\infty$ (see e.g. \cite[Lemma~3]{A-20}), the first assertion is proved.

%\end{proof}

In particular, we get the following new examples of r.i.\ spaces on $[0,1]$ that are not equivalent to any r.i.\ spaces on $(0,\infty)$.

\begin{corollary}\label{c4}
Let $0<\alpha\le 1$. Then, the Lorentz space $\Lambda_2(\log^{-\alpha}(e/u))$ has the following properties:

(a) any disjoint sequence in $\Lambda_2(\log^{-\alpha}(e/u))$ contains a subsequence $2$-equivalent to the unit vector basis of $\ell^2$, which spans a $2$-complemented subspace in $\Lambda_2(\log^{-\alpha}(e/u))$;

(b) Rosenthal's space ${{\mathcal{U}}}_{\Lambda_2(\log^{-\alpha}(e/u))}$ fails to be isomorphically embedded into $\Lambda_2(\log^{-\alpha}(e/u))$ and $\Lambda_2(\log^{-\alpha}(e/u))$ is not isomorphic to any r.i.\ space on $(0,\infty)$. 
\end{corollary}

\section{Existence of an isomorphic embedding $T\colon {\mathcal{U}}_E\to E$: the case when $T(\chi_{A_n})$, $n=1,2,\dots$, are independent.}

%%%%%%%%%%%%%%%%%%%%%%%%%%%%%%%%%%%%%

In the final section, we treat the special case when there is an isomorphic 
embedding $T\colon {\mathcal{U}}_E\to E$ such that the functions $T(\chi_{A_n})$, 
$n=1,2,\dots$, are independent symmetrically distributed r.v.'s.

Let $\{A_n\}_{n=1}^\infty$ be a sequence of disjoint measurable 
subsets of $(0,\infty)$ satisfying conditions \eqref{2}. In the same way as in the beginning of the proof of Theorem \ref{t3}, for every $m\in\N$, we find  pairwise disjoint sets $S^l_i\subseteq\N$, $i=1,2,\dots$, such that $\sum_{n\in S^l_i}m(A_n)=1/l$ and denote $B_i^l:=\bigcup_{n\in S_i^l}A_n$, $i=1,2,\dots$. 

Next, suppose that $E$ is a r.i.\ space such that ${\mathcal{U}}_E$ is isomorphically 
embedded into $E$, $T\colon {\mathcal{U}}_E\to E$ is an isomorphism, 
$y_i^l:=T(\chi_{B_i^l})$, $i,l\in\N$. In contrast to the preceding section, we  assume that sequences $\{y_i^l\}_{i=1}^\infty$, $l\in\N$, do not contain  ``almost'' disjoint subsequences, which means (see the proof of Theorem \ref{3})  
that $\|y_i^l\|_E\asymp \|y_i^l\|_{L^1}$, $i=1,2,\dots$, for each $l\in\N$.
Then, it is easy to check (see also \cite{KP}) that for every
$l\in\N$ there exists a constant $\varepsilon_l>0$ such that
$$
m\big(\{t:|y_i^l(t)|>\varepsilon_l\|y_i^l\|_E\}\big)\ge \varepsilon_l.
$$
However, we will need the following stronger condition: 
there are  $\alpha,\beta,\gamma>0$, an infinite sequence
$\{l_k\}_{k=1}^\infty\subset\N$, and a sequence of sets $F_k\subset\N$, $k=1,2,\dots$, such that $\gamma l_k\le \text{card }F_k\le l_k$ and for each $i\in F_k$ 
\begin{equation}\label{27}
m\big(\{t:|y_i^{l_k}(t)|>\alpha\})\ge \frac{\beta}{l_k}.
\end{equation}
%Clearly, we can assume that $\text{card }F_k\le  m_k$.

Furthermore, let us consider the family $\{B_i^{l_k},\,i\in F_k,k\in\N\}$. One can readily check now that definition of the sets $B_i^l$, $i,l\in\N$, and the conditions imposed on the sets $F_k$, $k\in\N$, assure that the latter family satisfies requirements \eqref{2}. Since Rosenthal's space ${\mathcal{U}}_E$ is invariant (up to isomorphism) on the particular choice of a sequence of sets satisfying \eqref{2} \cite[Lemma 8.7]{JMST}, without loss of generality, we can replace the initial sequence $\{A_n\}_{n=1}^\infty$ with the family $\{B_i^{l_k},\,i\in F_k,k\in\N\}$.

%%%%%%%%%%%%%%%%%%%%%%%%%%%%%%

\begin{theorem}\label{t5a}
Let  $E$ be a r.i.\ space on $[0,1]$ such that there exists an isomorphic embedding $T\colon {\mathcal{U}}_E\to E$. Suppose that there is a sequence
$\{l_k\}_{k=1}^\infty\subset\N$ such that the functions $y_i^{l_k}:=T(\chi_{B_i^{l_k}})$, $k,i\in\N$, are independent symmetrically distributed r.v.'s satisfying the  above conditions \eqref{27}. Then, the Kruglov operator $K$ is bounded from $E$ into $E''$.

Moreover, there is a constant $C>0$ such that
\begin{equation}\label{29}
\varphi_E\bigg(\Big(\frac{\beta}{2{l_k}}\Big)^{\gamma l_k}\bigg)\le \frac{C}{l_k},
\quad k=1,2,\dots,
\end{equation}
where $\varphi_E$ is the fundamental function of the space $E$.
\end{theorem}

%%%%%%%%%%%%%%%%%%%%%%%%%%%%%%

\begin{proof}
First, for each $k=1,2,\dots$, we compare the finite sequences $\{y_i^{l_k}\}_{i\in F_k}$ and $\{u_i^{l_k}\}_{i\in F_k}$, where $u_i^l$ are, as above, independent symmetrically distributed r.v.'s  equimeasurable with the characteristic functions $\chi_{B_i^{l_k}}$, $k,i=1,2,\dots$.
From \eqref{27} it follows that for all $\tau>0$
$$
m\big(\{t:|y_i^{l_k}(t)|>\tau\}\big)\ge 
\beta m\big(\{t:\alpha|u_i^{l_k}(t)|>\tau\}\big),\;\;i\in F_k,k=1,2,\dots
$$
Hence, applying the result of  Kwapien-Rychlik,
\cite[Ch.V, Theorem 4.4.]{VTCh}, for all $\tau>0$ and $a^k_i\in\mathbb{R}$, we get
$$
m\left(\left\{t:\left|\sum_{k=1}^\infty \sum_{i\in F_k} a^k_i u_i^{l_k}(t)\right|
>\tau\right\}\right)\le 
\frac{2}{\beta} 
m\left(\left\{t:\left|\sum_{k=1}^\infty \sum_{i\in F_k} a^k_i y_i^{l_k}(t)\right|
>\beta\alpha\tau\right\}\right).
$$
So, by \cite[Ch.II, \S4.3, Corollary 2]{krein-petunin-semenov},
$$
\left\|\sum_{k=1}^\infty \sum_{i\in F_k} a^k_i u_i^{l_k}\right\|_E
\le 
\frac{2}{\beta^2\alpha} 
\left\|\sum_{k=1}^\infty \sum_{i\in F_k} a^k_i y_i^{l_k}\right\|_E.
$$
On the other hand, since $T$ is an isomorphism, we have
\begin{equation}\label{28}
\left\|\sum_{k=1}^\infty \sum_{i\in F_k} a^k_i y_i^{l_k}\right\|_E
\mathrel{\substack{\textstyle \|T\| \\ \textstyle\asymp}}
\left\|\sum_{k=1}^\infty \sum_{i\in F_k} a^k_i \chi_{B_i^{l_k}}\right\|_{Z_E}.
\end{equation}
Combining the last inequalities, we infer that
%
%\begin{equation}\label{28}
$$
\left\|\sum_{k=1}^\infty \sum_{i\in F_k} a^k_i u_i^{l_k}\right\|_E
\le 
\frac{2\|T\|}{\beta^2\alpha} 
\left\|\sum_{k=1}^\infty \sum_{i\in F_k} a^k_i \chi_{B_i^{l_k}}\right\|_{Z_E}.$$
%\end{equation}
%
Applying now Theorem \ref{t1} (to the family $\{B_i^{l_k},\,i\in F_k,k\in\N\}$), we complete the proof of the first assertion.

Further, since $\text{card }F_k\le l_k$ and $m(B_i^{l_k})=1/l_k$, from \eqref{28} it follows that
$$
\Big\|\sum_{i\in F_k} y_i^{l_k}\Big\|_E\le C'\ \Big\| 
\sum_{i\in F_k} \chi_{B_i^{l_k}}\Big\|_E \le C',\quad k=1,2,\dots
$$
Moreover, taking into account the fact that $y_i^{l_k}$, $i\in F_k$,  are independent symmetrically distributed r.v.'s, the inequality $\text{card } F_k\ge \gamma l_k$ and \eqref{27}, we obtain
\begin{eqnarray*}
\Big\|\sum_{i\in F_k} y_i^{m_k}\Big\|_E &\ge& \alpha\gamma m_k\cdot
\big\|\chi_{\bigcap_{i\in F_k} \{y_i^{m_k}\ge\alpha\}}\big\|_E\\ &=& \alpha\gamma m_k\cdot \varphi_E\Big(\prod_{i\in F_k}m(\{y_i^{m_k}\ge\alpha\})\Big)\\
&\ge& \alpha\gamma m_k\cdot \varphi_E\Big(\big(\frac{\beta}{2{m_k}}\big)^{\gamma m_k}\Big).
\end{eqnarray*}
Combining these inequalities, we obtain \eqref{29}. 
\end{proof}

%%%%%%%%%%%%%%%%%%%%%%%%%%%%%%

\begin{corollary}\label{c4a}
Let $E$ be the exponential Orlicz space $\text{Exp}L^p$, $p>0$. 
Then, there exists an isomorphic embedding $T\colon {\mathcal{U}}_E\to E$,
satisfying the conditions of Theorem \ref{t5a}, if and only if $0<p\le1$. 
%Similar results hold also for the K\"othe duals $L'_{M_p}$.
\end{corollary}
%%%%%%%%%%%%%%%%%%%%%%%%%%%%%%

\begin{proof}
One can easily check that, for $E=\text{Exp}L^p$, we have $\varphi_E(u)\asymp \log^{-1/p}(e/u)$, $0<u\le1$. Therefore, a direct calculation shows that \eqref{29} is fulfilled in this case if and only if $0<p\le1$. Moreover, if $0<p\le1$, the space $\text{Exp}L^p$ has the Kruglov property (see
\cite[the beginning of \S2.4]{B} and \cite[4.3.1]{AS-10}), which implies that  
there exists an isomorphic embedding $T\colon {\mathcal{U}}_E\to E$,
satisfying the conditions of Theorem \ref{t5a} (indeed, we take $u_i^{m}$ for $y_i^{m}$, an arbitrary sequence of positive integers $\{m_k\}_{k=1}^\infty$ and any set of cardinality $m_k$ for $F_k$, $k=1,2,\dots$). Thus, the desired result follows.
\end{proof}

%%%%%%%%%%%%%%%%%%%%%%%%%%%%%%%%%%

%%%%%%%%%%%%%%%%%%%%%%%%%%%%%%%%%%

\end{document}